\documentclass[a4paper,11pt]{article}
\usepackage[T1]{fontenc}

\usepackage{tikz}
\usepackage[utf8]{inputenc}
\usepackage{amsmath,amsfonts,amsthm,amssymb, amsopn,mathtools}
\usepackage[final, pdftex, pdfpagelabels, pdfstartview = {FitH}, bookmarks, colorlinks, plainpages = false, linktoc=all, linkcolor=red, citecolor=blue, urlcolor=blue,filecolor=black]{hyperref}
\usepackage{cleveref}
\usepackage{pdfrender, tikz-cd, rotating}
\usepackage{fancyhdr}
\usepackage{amsmath,amsfonts}
\usepackage{amsthm}
\usepackage{amssymb, amsopn}
\usepackage[USenglish]{babel}
\usepackage{enumerate}
\usepackage{verbatim}
\usepackage{graphicx}
\usepackage[left=3cm, right=3cm, top=3cm, bottom=3cm]{geometry}
\usepackage{csquotes}
\usepackage{mathscinet} 
\usepackage{enumitem}

\usepackage[textwidth=3.5cm, color=orange!10]{todonotes}
\usetikzlibrary{matrix,arrows,decorations.pathmorphing,decorations.markings}

\theoremstyle{plain}
\newtheorem{thm}{Theorem}[section]
\newtheorem{lemma}[thm]{Lemma}
\newtheorem{prop}[thm]{Proposition}
\newtheorem{corollary}[thm]{Corollary}

\newtheorem{introthm}{Theorem}
\newtheorem{introprop}[introthm]{Proposition}

\theoremstyle{definition}
\newtheorem{definition}[thm]{Definition}
\newtheorem{example}[thm]{Example}
\newtheorem{remark}[thm]{Remark}

\theoremstyle{remark}
\newtheorem{claim}[thm]{Claim}

\renewcommand{\tilde}{\widetilde}
\renewcommand{\bar}{\overline}
\newcommand{\bbC}{\mathbb{C}}

\newcommand{\bbZ}{\mathbb{Z}}
\newcommand{\bbN}{\mathbb{N}}
\newcommand{\bbQ}{\mathbb{Q}}
\newcommand{\bbR}{\mathbb{R}}

\newcommand{\bbB}{\mathbb{B}}

\newcommand{\ra}{\rightarrow}
\newcommand{\xra}{\xrightarrow}
\newcommand{\cu}{\subseteq}

\newcommand{\affPhi}{\widehat{\Phi}}
\newcommand{\affW}{\widehat{W}}
\newcommand{\extW}{\affW_{ext}}
\newcommand{\affX}{\widehat{X}}
\newcommand{\affT}{\widehat{T}}

\newcommand{\sch}[1]{\bar{\Gr_{#1}}}

\newcommand{\calA}{\mathcal{A}}
\newcommand{\calB}{\mathcal{B}}

\newcommand{\Gr}{\mathcal{G}r}

\newcommand{\IC}{\mathrm{IC}}

\newcommand{\bfH}{\mathbf{H}}
\newcommand{\bfN}{\mathbf{N}}
\newcommand{\undH}{\underline{\mathbf{H}}}

\newcommand{\htil}{\tilde{h}}
\newcommand{\ftil}{\tilde{f}}
\newcommand{\etil}{\tilde{e}}
\newcommand{\strtil}{{\str}_{(2,1,2)}}

\newcommand{\eps}{\varepsilon}

\newcommand{\affHec}{\widetilde{\mathcal{H}}}
\newcommand{\undc}{\underline{c}}
\newcommand{\aand}{\qquad\text{ and }\qquad}
\newcommand{\Address}{
	\bigskip{\footnotesize
		
		\textsc{Dipartimento di Matematica, Universit\`a di Pisa, Italy}\par\nopagebreak
		\textit{E-mail address}: \texttt{leonardo.patimo@unipi.it}
}}

\DeclareMathOperator{\grdim}{grdim}

\DeclareMathOperator{\str}{string}
\DeclareMathOperator{\Arr}{Arr}
\DeclareMathOperator{\Ima}{Im}
\DeclareMathOperator{\wt}{wt}

\DeclareMathOperator{\Conv}{Conv}
\DeclareMathOperator{\HL}{HL}
\DeclareMathOperator{\rk}{rk}

\newcommand{\shift}{Z}

\setuptodonotes{size = small}
\title{The charge statistic in type $A$ via the Affine Grassmannian}
\author{Leonardo Patimo}

\begin{document}

\maketitle

\begin{abstract}
We give a new construction of Lascoux--Sch\"utzenberger's charge statistic in type A which is motivated by the geometric Satake equivalence. We obtain a new formula for the charge statistic in terms of modified crystal operators and an independent proof of this formula which does not rely on tableaux combinatorics.
\end{abstract}

\section*{Introduction}

Finite-dimensional representations of $SL_{n+1}(\bbC)$ are a classical topic in representation theory, with their structure and classification thoroughly explored. Irreducible representations are in bijection with partitions with at most $n$ parts $
\lambda=(\lambda_1\geq \lambda_2 \geq \ldots \geq \lambda_n\geq 0)$.
We denote by $L(\lambda)$ the irreducible representations of highest weight $\lambda$. The representations $L(\lambda)$  have a canonical basis, which can be constructed geometrically using the geometric Satake correspondence. This basis can be parametrized in terms of combinatorial object, such as semistandard Young tableaux or Littelmann's paths. Moreover, this canonical basis can  be endowed with an additional structure, known as \emph{crystal graph}, which serves as a combinatorial representation theoretic ``shadow'' of the full structure. We denote by $\calB(\lambda)$ the crystal graph of the representation $L(\lambda)$.
To compute the dimension of weight spaces of $L(\lambda)$ is thus sufficient to enumerate elements in a certain explicitly defined set. 

These weight multiplicities admit a $q$-analogue also known as the Kostka--Foulkes polynomials $K_{\lambda,\mu}(q)$. Kostka--Foulkes polynomials were originally defined in the context of symmetric functions, where they appear as the coefficients of the expansion of Schur polynomials in terms of Hall--Littlewood polynomials. Furthermore, they have a representation theoretic interpretation as they record the dimensions of the graded components of the Brylinski filtration on weight spaces \cite{BriLimits} and can also be obtained as Kazhdan--Lusztig polynomials for the corresponding affine Weyl group \cite{LusSingularities}. In particular, they have positive coefficients and this naturally leads to the search for combinatorial interpretations of their coefficients.

Lascoux and Sch\"utzenberger gave a combinatorial formula for the Kostka--Foulkes polynomials by defining a statistic on the set of semi-standard tableaux, called the \emph{charge statistic} \cite{LSSur}. This is a function $c:\calB(\lambda)_\mu\ra \bbZ_{\geq 0}$ such that
\begin{equation}\label{chargeintro}K_{\lambda,\mu}(q)=\sum_{T\in \calB(\lambda)_\mu} q^{c(T)}.  \end{equation}
This charge function is defined in terms of an operation, called \emph{cyclage}, that can be performed on the set of semistandard Young tableaux, after interpreting them as words in the plactic monoid.

Alternative constructions of the charge statistic are known.
In \cite{LLTCrystal}, Lascoux, Leclerc and Thibon gave a formula for the charge statistic in terms of the crystal graphs while Nakayashiki and Yamada \cite{NYKostka} defined the charge using the energy function on affine crystals. However, the proofs of all these formulas ultimately rely on tableaux combinatorics and are based on an operation called \emph{cyclage} of tableaux.

The goal of this papers is to construct a charge statistic on $\calB(\lambda)$ based on the geometric methods developed in \cite{ChargeGen}.
Motivated by the geometric Satake equivalence, we relate the charge statistic on $\calB(\lambda)$ to the geometry of the affine Grassmannian $\Gr^\vee$ of the Langlands dual group $PGL_n(\bbC)$. To our knowledge, this is the first time that the charge statistic is directly linked to the geometry of the affine Grassmannian.

\subsection*{Recharge statistics attached to family of cocharacters}

We now briefly recall the approach developed in \cite{ChargeGen}.
Let $T\subset SL_{n+1}(\bbC)$ denote the maximal torus consisting of diagonal matrices and let $T^\vee\subset PGL_{n+1}(\bbC)$ denote its dual torus. Let $\Gr^\vee$ denote the affine Grassmannian of $PGL_{n+1}(\bbC)$. Let $\affT:=T^\vee\times \bbC^*$ denote the augmented torus, which acts naturally on $\Gr^\vee$. 
 Each cocharacter $\eta\in X_\bullet(\affT)$ defines a $\bbC^*$-action on $\Gr^\vee$. This in turn defines, for each weight $\mu$, a hyperbolic localization functor $\HL^\eta_{\mu}$,  on perverse sheaves on $\Gr^\vee$ (cf. \cite[\S 2.4]{ChargeGen}).
 
 We consider a family of cocharacters $\eta(t)$. We assume that at $t=0$, $\eta(0)$ lies in the MV region \cite[Definition 2.18]{ChargeGen}. In the MV region the hyperbolic localization functors are exact, and are concentrated in a single degree when applied to intersection cohomology sheaves $\IC_\lambda$. Further, we assume that for $t\gg 0$ the cocharacter $\eta(t)$ lies in the KL region, where the graded dimension of hyperbolic localization computes the Kostka--Foulkes polynomials, i.e., we have \[K_{\lambda,\mu}(v^2)=\grdim \HL^{\eta(t)}_\mu(\IC_\lambda)v^{-\ell(\mu)}.\]
 
 Our goal is to find, for each cocharacter $\eta(t)$ in the family, a statistic $r(t,-)$, called \emph{recharge}, satisfying
 \begin{equation} \label{rstat}\grdim\left(\HL^{\eta(t)}_{\mu}(\IC_\lambda)\right)=\sum_{T\in \calB(\lambda)_\mu} q^{r(t,T)}\end{equation}
To achieve this, we start with a trivial statistic in the MV region and then modify it appropriately whenever our family $\eta(t)$ crosses a wall of the form
\[H_{\alpha^\vee}:=\left\{ \eta \in X_\bullet(\affT) \mid \langle \eta,\alpha^\vee\rangle =0\right\},\]
with $\alpha^\vee$ a real root of the affine root system of $\Gr^\vee$.

For each wall crossing, we determine a  \emph{swapping function} (cf. \cite[Definition 2]{ChargeGen}), which indicates how to modify the statistic $r$ in a way that \eqref{rstat} keeps being satisfied. Eventually, after doing this for all the walls crossed, we obtain a statistic $r(t,-)$ in the KL region, from which a charge statistic satisfying \eqref{chargeintro} can be easily deduced (cf. \eqref{charge}).



\subsection*{Crystals and twisted Bruhat graphs in type $A$}

To describe swapping functions in type $A$ we first need some work on crystal graphs and on the combinatorics of Bruhat graphs in type $A$.

Let $X$ be the character (or weight) lattice of $T$ and let $X_+$ the subset of dominant characters. 
For $\lambda \in X_+$, we consider the crystal $\calB(\lambda)$ associated to an irreducible representation of $G$ of highest weight $\lambda$. Let $f_i,e_i$, for $1\leq i \leq n$, denote the crystal operators on $\calB(\lambda)$.
 There is an action of $W$ on $\calB(\lambda)$, where the simple reflection $s_i$ acts by reversing each $f_i$-string.

Lecouvey and Lenart \cite{LLAtomic} defined an atomic decomposition of the crystal $\calB(\lambda)$ in type $A$. 
Our first result is an equivalent formulation of their decomposition (cf. \Cref{ccareLL}).
\begin{introthm}\label{introthm1}
	The closure of the $W$-orbits under the crystal operator $f_n$ are the atoms of an atomic decomposition of $\calB(\lambda)$.
\end{introthm}

 On $\calB(\lambda)$ we define crystal operators $e_\alpha,f_\alpha$ for any positive root $\alpha\in \Phi_+$ by twisting the usual crystal operators $f_i,e_i$ (see \Cref{mcoSec}). 
For any $\alpha\in \Phi_+$ and $T\in \calB(\lambda)$, let
$\eps_\alpha(T)=\max \{ k \mid e_\alpha^k(T)\neq 0\}$.
Then we define the \emph{atomic number} of $T$ as
\[ \shift(T)=\langle \wt(T),\rho^\vee\rangle+\sum_{\alpha \in \Phi_+} \eps_\alpha(T).\]

The name is explained by the main property of $Z$ (cf. \Cref{atomicthm}).
\begin{introthm}
	The atomic number $Z$ is constant on the atoms from \Cref{introthm1}.
\end{introthm}

We fix a family of cocharacters $\eta(t)$ going from the MV region to the KL region which crosses all the walls in a sort of lexicographic order (see \Cref{inductive}).
Whenever the family $\eta(t)$ crosses a wall $H_{\alpha^\vee}$, we change the orientation of the edges of the Bruhat graph labeled by $\alpha^\vee$. This procedure transforms the Bruhat graph into the \emph{twisted Bruhat graph} studied by Dyer \cite{DyeHecke}. The crucial observation is \Cref{Gammam}.

\begin{introprop}\label{propindegrees}
 Let $\lambda\in X_+$, $m\in \bbN$ and let $\Gamma^m_\lambda$ denote the twisted Bruhat graph whose vertices are the weight smaller than $\lambda$, obtained after changing the orientation of the edges corresponding to the first $m$ walls. Then the difference of the in-degrees  in $\Gamma_\lambda^m$ between the vertices of the edges that we reverse is always $1$.
\end{introprop}

\subsection*{Swapping functions and a new formula for the charge statistic}
 
 For the family of cocharacters $\eta(t)$ above, we can explicitly construct the swapping function in terms of the operators $e_\alpha$ (see \Cref{eswapping2} for a precise formulation). We denote by $<$ the Bruhat order on weights.

 \begin{introprop}
 Let $\alpha^\vee=c\delta-\beta^\vee$ the positive real root corresponding to a reflection $t$. Let $\mu\in X$ such that $\mu<t\mu\leq \lambda$. We define
 \begin{equation}\label{sfintro} \psi_{t\mu}:\calB(\lambda)_{t\mu}\ra 
 \calB(\lambda)_{\mu},\qquad
 \psi_{t\mu}(T)=
 e_\beta^{\langle \mu,\beta^\vee	\rangle +c}(T).\end{equation}
 Then, $\psi=\{\psi_\nu\}$ is a swapping function corresponding to the wall $H_{\alpha^\vee}$.
\end{introprop}

 To prove this, we restrict to a single atom and apply \Cref{propindegrees} after translating it into a problem on twisted Bruhat graphs.
 
We then use the swapping functions in \eqref{sfintro} to give a new explicit formula for the charge statistic on $\calB(\lambda)$. 
This is the main result of this paper (cf. \Cref{main}).
\begin{introthm}\label{intromain}
	The function 
\[ c(T) = Z(T)-\frac12\ell(\wt(T))\]
is a charge statistic on $\calB(\lambda)$, where $\ell$ denotes the Bruhat length on $X$. In particular, if $\wt(T)$ is a dominant weight, we have 
\[ c(T) = \sum_{\alpha\in \Phi_+}\eps_\alpha(T). \]

\end{introthm}

This gives a very concise description of the charge. However, 
we regard as the most interesting aspect of this result the fact that this formula is established by exploiting the connection with the geometry of the affine Grassmannian.  
We remark in fact that our proof of \Cref{intromain} is independent from previous proofs \cite{LSSur,LLTCrystal,NYKostka} (for example, the cyclage operation on tableaux does not play any role here). Nevertheless, in \Cref{chargecoincide} we check that our charge statistic coincides with Lascoux--Sch\"utzenberger's one. 

\subsection*{Open questions and future work}

In a joint project with J. Torres, we developed a similar strategy in type $C$, determining an atomic decomposition and defining swapping functions. This work led to a new formula for the charge statistic in type $C$ in rank $2$ \cite{PTAtoms}, a case for which a charge was not known previously.
 Further investigations into type $C$ in higher ranks are currently ongoing.

We hope these methods will shed light on the combinatorics of Kazhdan--Lusztig polynomials in affine type. Even in type $A$, we are only able to produce a charge statistic starting from a specific family of cocharacters. Considering other families might provide further insights into the combinatorics of Kazhdan–Lusztig polynomials in affine type. Additionally, it would be interesting to explore connections between these combinatorial interpretations and the new combinatorial formulas for Kazhdan–Lusztig polynomials proposed in \cite{PatCombinatorial} and \cite{BBD+Towards}.

Although they are motivated and informed by the geometry of the affine Grassmannian, the results in this paper are stated in purely combinatorial terms, relying on the structure of the crystal of the representation. It seems natural to lift this process to a more  geometric level, identifying an analogue of the swapping function for MV cycles and/or MV polytopes. We believe that, by applying an analogue of the crystalline procedure outlined in this paper, it should be possible to construct distinguished bases for the intersection cohomology of Schubert varieties in the affine Grassmannian. Such an approach would yield results analogous to those obtained in \cite{PatBases} for the finite Grassmannian.

\subsection*{Acknowledgments} 
I am grateful to Mark Shimozono for pointing out the similarity and the possible connections with Nakayashiki--Yamada's description of the charge in terms of the energy function \cite{NYKostka}.

%

\section{Crystals and the atomic decomposition}\label{CrystalSection}

In this section, we recall the the atomic decomposition of crystals of type $A$ studied by Lecouvey and Lenart in type $A$ \cite{LLAtomic}. While they work only on the dominant part $\calB_+(\lambda)$ of the crystal, we use a version of the decomposition on the whole crystal, which has a particularly simple expression and it is easier to work with. We also introduce a function on the crystal, called the atomic number $Z$, which is constant on the atoms and that will play an important role in our construction of the charge.

\subsection{Notation}

Let $G$ be the reductive group $SL_{n+1}(\bbC)$ and let $T$ be the maximal torus consisting of diagonal matrices. We denote by $X$ the  character lattice of the torus  $T$, which is isomorphic to $\bbZ^{n}$. The elements of $X$ are called \emph{weights}. 
Let $\Phi\subset X$ be the root system of $G$. It is a root system of type $A_n$. Let $W$ denote its Weyl group, which is isomorphic to the symmetric group $S_n$. We denote by $\alpha_1,\ldots,\alpha_n$ its simple roots and write 
\[ \alpha_{j,k}:=\alpha_j+\ldots +\alpha_k.\]
Let $X_+\subset X$ be the subset of dominant weights.

For $\lambda\in X_+$ we denote by $L(\lambda)$ the corresponding irreducible representation of highest weight $\lambda$ and by $\calB(\lambda)$ the corresponding crystal graph.
There are several equivalent different constructions of $\calB(\lambda)$. 
For $T\in \calB(\lambda)$, we denote by $\wt(T)\in X$ the weight of $T$.

The crystal $\calB(\lambda)$ can be constructed algebraically using the quantum group of $G$ \cite{KasCrystal,LusCanonical}. A geometric construction of $\calB(\lambda)$ can be obtained via geometric Satake \cite{BGCrystals,BFGUhlenbeck}. In this case the elements of the crystal are the MV cycles and the crystal operators are induced by a geometric operation on cycles. It is shown in \cite{KamCrystal} that the algebraic and geometric constructions are equivalent. The crystal $\calB(\lambda)$ can also be constructed combinatorially in terms of semistandard Young tableaux.

\subsection{Modified Crystal Operators}\label{mcoSec}

For $1\leq i \leq n$, let $e_i$ and $f_i$ be the crystal operators on $\calB(\lambda)$.
 We call \emph{$\alpha_i$-string} of $T$ the set $\{ f_i^k(T)\}_{k\geq 0} \cup \{ e^k_i(T)\}_{k\geq 0}$.

On $\calB(\lambda)$ there exists an action of the Weyl group with the simple reflection $s_i$ acting by reversing the $\alpha_i$-strings. We have $\wt(w\cdot T)=w\cdot \wt(T)$.
 Using the action of $W$, we can associate, similarly to \cite{LLAtomic}, crystal operators to any positive roots.

\begin{definition}
	Let $\alpha=\alpha_{j,k}\in \Phi_+$ be a positive root. The \emph{modified crystal operators} $f_\alpha$ and $e_\alpha$ are defined respectively as $wf_kw^{-1}$ and $w e_k w^{-1}$, where $w=s_js_{j+1}\ldots s_{k-1}$.
\end{definition}
\begin{example}
	For any $i$, we have $f_{\alpha_i+\alpha_{i+1}}=s_if_{i+1}s_i$.
\end{example}

If $f_{\alpha}(T)\neq 0$, we have $\wt(f_{\alpha}(T))=\wt(T)-\alpha$. 
Moreover, we can have $f_{\alpha}(T)=0$ only if $\langle \wt(T),\alpha^\vee\rangle \leq 0$.

\begin{definition}	
For any $T\in \calB(\lambda)$ we define \[\phi_\alpha(T):=\max \{ N \mid f_\alpha^N(T)\neq 0\}\qquad \text{and}\qquad \eps_\alpha(T) := \max \{N \mid e_\alpha^N(T)\neq 0\}.\]
	 We call \emph{$\alpha$-string} of $T$ the set $\{ f_\alpha^k(T)\}_{k\geq 0} \cup \{ e^k_\alpha(T)\}_{k\geq 0}$.
\end{definition}

For any $i$, we have $f_i=f_{\alpha_i}$ so we simply write $\phi_i$ and $\eps_i$ instead of $\phi_{\alpha_i}$ and $\eps_{\alpha_i}$.
 If $\alpha=\alpha_{j,k}$ we have $\phi_{\alpha}=\phi_k \circ w^{-1}$, where $w=s_j\ldots s_{k-1}$.
Moreover, we have $\phi_\alpha= \eps_\alpha \circ s_\alpha$ and 
\begin{equation}\label{phi-eps}
\phi_\alpha(T)-\eps_\alpha(T)=\langle \wt(T), \alpha^\vee\rangle.
\end{equation}
Notice that if $f_\alpha=wf_k w^{-1}$, then $s_\alpha=w s_k w^{-1}$, so $s_\alpha(T)$ belongs to the $\alpha$-string of $T$.

\begin{remark}
	Let $\alpha=\alpha_{j,k}$. Then $f_\alpha=vf_kv^{-1}$ for any $v\in \langle s_1,\ldots ,s_k\rangle$ such that $v(\alpha_k)=\alpha$ (cf. \cite[Remark 5.1]{LLAtomic}). In fact, if $w=s_js_{j+1}\ldots s_{k-1}$ then $v^{-1}w\in \mathrm{Stab}(\alpha_k)\cap \langle s_1,\ldots,s_k\rangle$, and therefore $v^{-1}w\in \langle s_1,\ldots ,s_{k-2}\rangle$.
	Then we have $v^{-1}f_\alpha v=(v^{-1}w)f_k(v^{-1}w)^{-1}=f_k$ since the operator $f_k$ commutes with $f_j$ (and thus with $s_j$) for any $j$ such that $\langle \alpha_j,\alpha_k^\vee\rangle =0$. 
\end{remark}

The following computation about $\alpha$-strings in rank $2$ will be needed later.
\begin{lemma}\label{21explicit}
	For any $i$ with $1\leq i\leq n-1$ the function $\eps_i+\phi_{i+1}$ is constant along $(\alpha_i+\alpha_{i+1})$-strings.
\end{lemma}
\begin{proof}
	After Levi branching to $\langle \alpha_{i},\alpha_{i+1}\rangle \cu \Phi$, we can assume that $\calB(\lambda)$ is a crystal of type $A_2$. Let $\lambda=N\varpi_1+ N'\varpi_2$, with $N,N'\geq 0$. Recall from \cite[\S 11.1]{BSCrystal} that in type $A_2$ any $T\in \calB(\lambda)$ is uniquely determined by a triple of positive integers $\str(T):=(a,b,c)$ where $T=f_1^af_2^bf_1^c v$ and $v\in \calB(\lambda)$ is the unique element of weight $\lambda$.
	
	Let $T\in \calB(\lambda)$ be such that $\str(T)=(a,b,c)$. From \cite[eq. (11.10)]{BSCrystal} we have $\str(s_1(T))=\sigma_1(a,b,c)$, where \[\sigma_1(a,b,c):=(N+b-a-2c,b,c),\] and $\str(f_2(T))=\theta p_1 \theta(a,b,c)$, where \[\theta(a,b,c):=(\max(c,b-a),a+c,\min(b-c,a))\] and $p_1(a,b,c)=(a+1,b,c)$.
	From this it follows that 
	\[\str(f_{\alpha_1+\alpha_2}(T))=\sigma_1\theta p_1\theta\sigma_1(\str(T)).\] 
	We claim that
	\[\str(f_{\alpha_1+\alpha_2}(T))=\begin{cases}
	(a+1,b+1,c) & \text{if }a+c\geq N\\
	(a,b+1,c+1) & \text{if }a+c< N.
	\end{cases}\]
	Assume that $a+c\geq N$. Let $d:=N+b-a-2c$, so $d\leq b-c$. Then, we have
	\begin{align*}
	p_1\theta \sigma_1 (a,b,c)=&(\max(c,b-d)+1,d+c,\min(b-c,d))=\\
	=&\left(\max(c,b-d+1),d+c,\min(b-c+1,d)\right)= \theta \sigma_1(a+1,b+1,c)
	\end{align*}
	and the claim follows.	The case $a+c<N$ is similar.
	
	To conclude the proof we need to evaluate $\eps_1+\phi_2$ on the $(\alpha_1+\alpha_2)$-string of $T$. We have $\eps_1(T)=a$ and 
	\[\phi_2(T)=\eps_2\sigma_2\theta(a,b,c)=N'+a+c-\max(c,b-a)-2\min(b-c,a).\]
	
	A simple computation allows us to conclude by checking that, for $T',T''\in \calB(\lambda)$ such that $\str(T')=(a+1,b+1,c)$ and $\str(T'')=(a,b+1,c+1)$, we have 
	\[(\eps_1+\phi_2)(T')=(\eps_1+\phi_2)(T)=(\eps_1+\phi_2)(T'').\qedhere\]	
\end{proof}

\subsection{The Atomic Decomposition of Crystals}
 Let $\lambda \in X_+$ and let $\calB(\lambda)$ denote the corresponding crystal. We recall from \cite{LLAtomic} 
 the definition of atomic decomposition.
 
 \begin{definition}
 	An \emph{atomic decomposition} of a crystal $\calB(\lambda)$ is a partition of its elements $\calB(\lambda)= \bigsqcup \calA_i$ such that each $\calA_i$ satisfies the following conditions. 
 	\begin{itemize}
 		\item $\wt(T)\neq \wt(T')$ for any $T,T'\in \calA_i$ such that $T\neq T'$,
 		\item the weights in $\calA_i$ form a lower interval in the Bruhat order, i.e. there exists $\mu \in X_+$ such that 
 		\[\{ \wt(T) \mid T \in \calA_i\}=\{ \nu \in X \mid \nu\leq \mu\}.\] 
 	\end{itemize}
		In this case, we say that $\calA_i$ is an \emph{atom of highest weight} $\mu$.
 \end{definition}
 
 Lecouvey and Lenart show that in type $A$ every crystal $\calB(\lambda)$ admits an atomic decomposition. Moreover, they provide an explicit construction of this decomposition that we now recall. They use a different version of modified crystal operators. 
 For $\alpha\in \Phi_+$ they define $\ftil_\alpha=w f_n w^{-1}$ and $\etil_\alpha=w e_n w^{-1}$ where $w \in W$ is such that $w(\alpha_n)=\alpha$. This is well defined by \cite[Remark 5.1]{LLAtomic}. In general, we have $f_\alpha\neq \tilde{f}_\alpha$ but $f_\alpha=\ftil_\alpha$ for $\alpha=\alpha_{k,n}$.
 We write $\ftil_i$ and $\etil_i$ resp. for $\ftil_{\alpha_i}$ and $\etil_{\alpha_i}$.

\begin{definition}\label{B+}
 	We define a directed graph $\bbB^+(\lambda)$ as follows
 \begin{itemize}
 	\item the vertices of $\bbB^+(\lambda)$ are the elements $T\in \calB(\lambda)$ such that $\wt(T)\in X_+$,
 	\item there is an edge $T\ra T'$ if $\tilde{f}_\alpha(T')=T$ for some $\alpha \in \Phi_+$.
 \end{itemize}	
 \end{definition}

Notice that \Cref{B+} is slightly different from the original definition of $\bbB^+(\lambda)$ given in \cite[\S 6]{LLAtomic}. Lecouvey and Lenart impose in fact an extra condition: they require that there is an edge in $\bbB^+(\lambda)$ between $T$ and $T'$ if and only if $\tilde{f}_\alpha(T')=T$ \emph{and} $\wt(T') \gtrdot \wt(T)$.\footnote{We use the notation $\mu \gtrdot \mu'$ for the covering relation in the dominance order of $X_+$. This means that $\mu > \mu'$ and there is no $\nu\in X_+$ with $\mu>\nu>\mu'$} However, since we are only working in type $A$ and we are only interested in the connected components of $\bbB^+(\lambda)$, this extra condition turns out to superfluous. In fact, as the second part of the next Lemma shows, two vertices are connected in $\bbB^+(\lambda)$ if and only if they are connected in Lecouvey--Lenart's version of $\bbB^+(\lambda)$.\footnote{This is however only true in type $A$. In other types, the two definitions may lead to graphs with different connected components, see \cite[Remark 6.1]{LLAtomic}.}
\begin{lemma}\label{abcommute}
	Let $\alpha,\beta,\gamma\in \Phi_+$ and $T,T'\in \calB^+(\lambda)$.
	\begin{enumerate}
		\item Assume that $\gamma=\alpha+\beta$,  that $\langle \wt(T),\alpha^\vee\rangle>0$ and  $\langle \wt(T),\beta^\vee\rangle>0$. Then 
	\begin{equation}\label{abc}\ftil_\gamma(T)=\ftil_{\beta}\ftil_{\alpha}(T)=\ftil_{\alpha}\ftil_{\beta}(T)\neq 0.
	\end{equation} 
	\item If $T\ra T'$ is an edge in $\bbB^+(\lambda)$ with $\tilde{f_\gamma}(T')=T$, then there exist $\beta_1,\ldots,\beta_r\in \Phi_+$ such that 
	\begin{enumerate}
		\item  $\gamma=\beta_1+\ldots +\beta_r$,
		\item $\ftil_{\beta_i}\ldots \tilde{f}_{\beta_1}(T')\in \bbB^+(\lambda)$ for any $i\leq k$ and $\tilde{f_\gamma}(T')=\tilde{f}_{\beta_r}\ldots \tilde{f}_{\beta_1}(T')=T$,
		\item $\wt(T') \gtrdot \wt(\ftil_{\beta_1}(T'))\gtrdot \wt(\ftil_{\beta_2}\ftil_{\beta_1}(T'))\gtrdot \ldots \gtrdot \wt(T)$.
	\end{enumerate}
\end{enumerate}
\end{lemma}
\begin{proof}
	The first part is case (1.i) of \cite[Theorem 5.3]{LLAtomic}.
	
	We prove the second part by induction on $\langle \gamma,\rho^\vee\rangle$.
	Notice that if $\gamma=\alpha+\beta$ with $\alpha,\beta\in \Phi^+$, then $\langle \wt(T'),\alpha^\vee\rangle \geq \langle \gamma,\alpha^\vee\rangle >0$ and $\langle \wt(T'),\beta^\vee\rangle>0$.  In view of the first part, it is enough to show that, if $\wt(T')$ does not cover $\wt(T)=\wt(T')-\gamma$, then there exists $\alpha,\beta\in \Phi_+$ with $\alpha+\beta=\gamma$ such that $\wt(T')-\alpha\in X_+$.
	
	Assume that $\wt(T')$ does not cover $\wt(T)=\wt(T')-\gamma$.
	Then there exist a positive combination of simple roots $\zeta\in \bbN \Delta$ such that $\wt(T')\gtrdot\wt(T')-\zeta>\wt(T)-\gamma$. 
By \cite[Theorem 2.6]{StePartial} we have $\zeta\in \Phi_+$, and $\zeta<\gamma$ because $\wt(T')-\zeta>\wt(T)-\gamma$. In particular, $\gamma$ is not simple and we can assume $\gamma=\alpha_{j,k}$ with $j<k$. Then we have $\zeta=\alpha_{j',k'}$ for some $j\leq j'\leq k'\leq k$. We can further assume $k'\neq k$ (the case $j'\neq j$ is similar). Since $\wt(T)\in X_+$ and $\wt(T')-\alpha_{j',k'}=\wt(T)+\alpha_{j,j'-1}+\alpha_{k'+1,k}\in X_+$, we also have $\wt(T)+\alpha_{k'+1,k}=\wt(T')-\alpha_{j,k'}\in X_+$. In fact,
\[\langle \wt(T)+\alpha_{k'+1,k},\alpha_i^\vee\rangle=\begin{cases}
\langle \wt(T),\alpha_i^\vee\rangle & \text{if }i< k'\\
\langle \wt(T)+\alpha_{j,j'-1}+\alpha_{k'+1,k},\alpha_i^\vee\rangle-\langle \alpha_{j,j'-1},\alpha_k^\vee\rangle & \text{if }i\geq  k'
\end{cases}  \]
is positive for any $i$.
 We conclude since $\gamma=\alpha_{j,k'}+\alpha_{k'+1,k}$.
\end{proof}

\begin{proof}[Alternative proof of \Cref{abcommute}(1)]
	
	Here we sketch an alternative proof for the first part which does not rely on the combinatorics of semistandard tableaux. 
	
	We can find $w\in W$ such that $w(\alpha)=\alpha_{n-1}$ and $w(\beta)=\alpha_{n}$. So we can restrict to a crystal of type $A_2$ and it is enough to prove the statement for $\alpha=\alpha_1$ and $\beta=\alpha_2$.
	Assume $\lambda=N\varpi_1+N'\varpi_2$ with $N,N'>0$.

	Similarly to \Cref{21explicit}, we can associate to any element $B\in \calB(\lambda)$ in a unique way a triple of integers $\strtil(B)=(a,b,c)$ such that $B=f_2^af_1^bf_2^c(v)$ where $v\in \calB(\lambda)$ is the unique element of maximal weight.
	
	Let	$\strtil(B)=(a,b,c)$.
	If $f_2(B)=\ftil_2(B)\neq 0$, we have $\strtil(\ftil_2(B))=p_1(a,b,c)=(a+1,b,c)$. 
	We have $\ftil_{1}=s_2s_1p_1s_1s_2$ and we can explicitly compute the action of $\ftil_1$ on strings as $\sigma_2 \theta\sigma_1\theta p_1\theta \sigma_1\theta \sigma_2$. As a result, if $\ftil_1(B)\neq 0$ we obtain
		\begin{equation}\label{f2} \strtil(\ftil_1(B))=\begin{cases}
		(a,b+1,c) & \text{if }b< N' \\
		(a-1,b+1,c+1) & \text{if }b\geq  N'. 
		\end{cases}\end{equation}
	
		Let $\ftil_{12}:=\ftil_{\alpha_1+\alpha_2}$.
	We can compute $\ftil_{12}$ on strings as $\theta \sigma_1\theta f_2\theta\sigma_1\theta$. If $\ftil_{12}(B)\neq 0$, we obtain
\begin{equation}\label{f12} \strtil(\ftil_{12}(B))=\begin{cases}
	(a+1,b+1,c) & \text{if }b< N' \\
	(a,b+1,c+1) & \text{if }b\geq  N'. \\	
\end{cases}\end{equation}

Let now $T$ be as in the statement. Since $\langle \wt(T),\alpha_1^\vee\rangle >0$, we have $\ftil_1(T)\neq 0$. Moreover,  $\langle \wt(\ftil_1(T)),\alpha_2^\vee\rangle >0$, so also $\ftil_2\ftil_1(T)\neq 0$. In the same way we see that all the terms occurring in \eqref{abc} are not trivial. The equality follows now from the explicit expressions \eqref{f2} and \eqref{f12}.
\end{proof}

For any element $T\in \calB(\lambda)$ we denote by $\bar{T}$ the unique element in the orbit $W\cdot T$ such that $\wt(\bar{T})\in X_+$.

\begin{definition}
	We say that $T,T'\in \calB(\lambda)$ are \emph{LL-equivalent} if $\bar{T}$ and $\bar{T'}$ lie in the same connected component of $\bbB^+(\lambda)$.

We call \emph{LL atoms} the equivalence classes of $\calB(\lambda)$ with respect to the LL equivalence. 
\end{definition}

Then \cite[Theorem 6.5]{LLAtomic} 
 can be easily reformulated as follows.
\begin{thm}\label{LLtheorem}
	The LL atoms form an atomic decomposition of $\calB(\lambda)$.
\end{thm}

\begin{remark}
	To prove \Cref{LLtheorem}, Lecouvey and Lenart make use of the combinatorics of semistandard tableaux and of the cyclage operations (in \cite[Lemma 5.17 and 5.18]{LLAtomic}). This can be avoided, and we briefly explain here how one could achieve a proof independent from \cite{LSSur}. 
	
	A crucial point in Lecouvey and Lenart's proof is to to show that the operators $\tilde{f}_\alpha$ and $\tilde{f}_\beta$, as well as their inverses $\etil_\alpha$ and $\etil_\beta$, commute in $\calB^+(\lambda)$ (cf. \cite[Theorems 5.3 and 5.5]{LLAtomic}). After conjugation by an element of $W$, this can be reduced to show the following four commutativity statements \textbf{C1-4} (cf. \cite[Lemma 5.4 and 5.6]{LLAtomic}).
	\begin{enumerate}[label=\textbf{C\arabic*.}]
		\item \label{C1} Let $\calB(\lambda)$ be a crystal of type $A_2$ and let $T\in \calB^+(\lambda)$ be such that $\langle \wt(T),\alpha_1^\vee\rangle>0$ and $\langle \wt(T),\alpha_2^\vee\rangle>0$, then $\ftil_1\ftil_2(T)=\ftil_2\ftil_1(T)=\ftil_{\alpha_1+\alpha_2}(T)\neq 0$.

		\item \label{C2} Let $\calB(\lambda)$ be a crystal of type $A_2$  and let $T\in \calB^+(\lambda)$ be such that $\etil_1(T)\neq 0$ and $\etil_2(T)\neq 0$, then \begin{equation}\label{e12=e21}
		\etil_1\etil_2(T)=\etil_2\etil_1(T)=\etil_{\alpha_1+\alpha_2}(T)\neq 0.
		\end{equation}

		
		\item \label{C3} Let $\calB(\lambda)$ be a crystal of type $A_3$ and let $T\in \calB(\lambda)$ be such that $\langle \wt(T),\alpha_1^\vee\rangle>0$ and $\langle \wt(T),\alpha_3^\vee\rangle>0$. Then $\ftil_1\ftil_3(T)=\ftil_3\ftil_1(T)\neq 0$.
		

		\item \label{C4} Let $\calB(\lambda)$ be a crystal of type $A_3$ and let $T\in \calB^+(\lambda)$ be such that $\etil_1(T)\neq 0$, $\etil_3(T)\neq 0$ and $\langle\wt(T),\alpha_2\rangle\geq 1$, then \begin{equation}\label{e13=e31}
			\etil_1\etil_3(T)=\etil_3\etil_1(T)\neq 0.
		\end{equation}


\end{enumerate}	
	
The equality \ref{C1} is shown in \Cref{abcommute}(1). To obtain \ref{C2}, notice that $\etil_2(B)=0$ if and only if $\strtil(B)=(0,b,c)$. Then, it immediately follows from \eqref{f2} that $\etil_2\etil_1(T)\neq 0$. The claim now follows by applying \Cref{21explicit} to $T':=\etil_2\etil_1(T)$.

To show \ref{C3} and \ref{C4} one can proceed in a way similar to the proof of \Cref{21explicit}, i.e., one can compute explicitly the action of $\ftil_1$ on string patterns in rank $3$, with respect to the reduced expression $w_0=s_3s_2s_3s_1s_2s_3$. This is a tedious but straightforward computation, and it will show that $\ftil_1\ftil_3(T)=\ftil_3\ftil_1(T)$ as long as both terms are well-defined and non-zero. Moreover, if $T\in \calB^+(\lambda)$ is such that $\eps_3(T)>0$, then also $\eps_3(\etil_1(T))>0$. In particular, we obtain that $\etil_1\etil_3(T)\neq 0$ if $\etil_3(T)\neq 0$ and $\etil_3(T)\neq 0$. Finally, \ref{C4} follows by applying \ref{C3} to $T':=\etil_1\etil_3(T)\neq 0$.

	

\end{remark}
 
 The goal of this section is to give an equivalent construction of the LL atoms which turns out to be more convenient for our purposes.
 We start by defining another graph.
 
 \begin{definition}
 	We define $\bbB_1(\lambda)$ to be the graph with vertices the elements of $\calB(\lambda)$ and such that there is an edge between $T$ and $T'$ if and only if $T=w(T')$ for some $w\in W$ or $T=f_n(T')$.
 \end{definition}

\begin{prop}\label{ccareLL}
	The connected components of $\bbB_1(\lambda)$ are the LL atoms of $\calB(\lambda)$.
\end{prop}

In the proof we make use of the following elementary Lemma about weights in type $A$.
\begin{lemma}\label{Aweight}
	Let $\mu \in X$. Then, either $s_n(\mu)=\mu+\alpha_n$ or there exists $w\in W$ such that both $w(\mu)$ and $w(\mu+\alpha_n)$ lie in $X_+$.
\end{lemma}
\begin{proof}
	Let $x\in W$ be such that $x(\mu)\in X_+$ and assume that $x(\mu)+x(\alpha_n)\not\in X_+$. This means that there exists $i$ such that 
\[\langle x(\mu),\alpha_i^\vee\rangle\geq 0\aand \langle x(\mu)+x(\alpha_n),\alpha_i^\vee\rangle <0.\]

	There are two cases to consider: either
	$\langle x(\mu),\alpha_i^\vee\rangle=1$ and 
	$\langle x(\alpha_n),\alpha_i^\vee\rangle = -2$ or $\langle x(\mu),\alpha_i^\vee\rangle=0$ and 	
	 $\langle x(\alpha_n),\alpha_i^\vee\rangle <0$.

	Suppose we are in the first case. Then we must have $x(\alpha_n)=-\alpha_i$ and $x(\alpha_n^\vee)=-\alpha_i^\vee$. This means that 
	\[ \langle \mu,\alpha_n^\vee\rangle = \langle x(\mu),x(\alpha_n^\vee)\rangle=-1\]
	and $s_n(\mu)=\mu+\alpha_n$.

	We can assume now that $\langle x(\mu),\alpha_i^\vee\rangle=0$ and 	
	$\langle x(\alpha_n),\alpha_i^\vee\rangle <0$. 
	By replacing $x$ with $x'=s_ix$, we obtain $x'(\mu)=x(\mu)\in X_+$ and $x'(\mu+\alpha_n)=x(\mu)+x(\alpha_n)+c\alpha_i$, with $c>0$. We can repeat this if necessary (that is, if $x(\mu)+x(\alpha_n)+c\alpha_i\not \in X_+$): it will eventually terminate since any increasing sequence of roots is finite. At the end, either $s_n(\mu)=\mu+\alpha_n$ or we find $w\in W$ such that $w(\mu)=x(\mu)\in X_+$ and $w(\mu+\alpha_n)\in X_+$.
\end{proof}

\begin{proof}[Proof of \Cref{ccareLL}.]
	Let $T,T'\in \calB(\lambda)$. We assume first that $T$ and $T'$ are in the same connected component of $\bbB_1(\lambda)$. We want to show that $T$ and $T'$ are in the same LL atom.
	
	If $T=w(T')$, then by definition $\bar{T}=\bar{T'}$ and they are trivially in the same LL atom. Therefore, it is sufficient to consider the case $T=f_n(T')$.

If $s_n(\wt(T))=\wt(T)+\alpha_n=\wt(T')$, then also $s_n(T)=T'$. So we can assume $s_n(\wt(T))\neq \wt(T)+\alpha_n$. 
Then, by \Cref{Aweight},
there exists $w\in W$ such that $w(T)=\bar{T}$ and $w(T')=\bar{T'}$.
We have \[w f_n w^{-1}(\bar{T'})=\bar{T}.\]
If $w(\alpha_n)\in \Phi_+$ then $w f_n w^{-1}=\ftil_{w(\alpha_n)}$. Otherwise, we have
$\ftil_{-w(\alpha_n)}=ws_n f_ns_n w^{-1}=we_n w^{-1}$, hence $\ftil_{-w(\alpha_n)}(\bar{T})=\bar{T'}$. In both cases, $T$ and $T'$ clearly lie in the same LL atom.

	We assume now that $T,T'$ are in the same LL atom. We can assume $T,T'\in \bbB^+(\lambda)$ and that $T=\ftil_{\alpha} (T')$ for some $\alpha\in \Phi_+$. There exists $w\in W$ such that $\alpha=w(\alpha_n)$ and $\ftil_{\alpha}=w f_n w^{-1}$. We conclude since
	 the connected components of $\bbB_1(\lambda)$ are closed under the action of $W$ and of $f_n$.
\end{proof}

The LL atomic decomposition is preserved by some modified crystal operators.
\begin{lemma}\label{fandatoms}
	Let $\lambda\in X_+$ and $\alpha=\alpha_{j,n}\in \Phi_+$ with $j\leq n$. Then $f_{\alpha}$ and $e_\alpha$ preserve the LL atoms of $\calB(\lambda)$.
	If $T$ is contained in a LL atom $\calA$ of highest weight $\lambda'\in X_+$, we have $f_{\alpha}^k(T)\neq 0$ if and only if $\wt(T)-k\alpha\leq \lambda'$. 
\end{lemma}
\begin{proof}
	We have $f_{\alpha}=wf_nw^{-1}$ for some $w\in W$, so $f_{\alpha}$ preserves the atoms since both $f_n$ and $W$ do. The same holds for $e_\alpha$.
	
	It is enough to show the second statement for $k=1$.
	If $f_{\alpha}(T)\neq 0$, then $\wt(T)-\alpha$ is a weight in $\calA$, hence $\wt(T)-\alpha\leq \lambda'$.
	
	On the other hand, if $\wt(T)-\alpha\leq \lambda'$, then there exists $T'\in \calA$ such that $\wt(T')=\wt(T)-\alpha$.
	If $\langle \wt(T),\alpha^\vee\rangle> 0$ then $f_{\alpha}(T)\neq 0$. If $\langle \wt(T),\alpha^\vee\rangle\leq 0$, then
	$\langle \wt(T'),\alpha^\vee\rangle< 0$ and so $e_{\alpha}(T')\neq 0$. In both cases, it follows that $e_\alpha(T')=T$ and $f_\alpha(T)=T'\neq 0$.
\end{proof}

\begin{remark}
	It is immediate from the definition that the operators $\ftil_\alpha$ always preserve the LL atoms. In particular, since every weight in an atom has multiplicity one, we see that for any $\alpha,\beta\in \Phi_+$ and $T\in \calB(\lambda)$, we have  $\ftil_\alpha\ftil_\beta(T)=\ftil_\beta\ftil_\alpha(T)$ as long as they are both well-defined and non-zero.

	On the other hand, in general, the operators $f_\alpha$ do not preserve the atomic decomposition. This hints to the following  interesting generalization of $\bbB_1(\lambda)$.
	
	 For every $i$, we can construct a graph $\bbB_i(\lambda)$ with the elements of $\calB(\lambda)$ as vertices and where there is an edge between $T$ and $T'$ if and only if $T=w(T')$ for some $w\in W$ or $T=f_{n-j}(T')$ for $j<i$. Clearly, the connected components of $\bbB_i(\lambda)$ are union of connected components of $\bbB_{i-1}(\lambda)$, so in this way we obtain a sequence of partitions of $\calB(\lambda)$, each coarser than the previous one. These connected components of $\bbB_i(\lambda)$ seem in many cases to be related to the $i$-th pre-canonical basis introduced in \cite{LPPPre}.
\end{remark}

\subsection{The Atomic Number}

\begin{definition}\label{atomicdef}
	For $T\in \calB(\lambda)$ we define the \emph{atomic number} $\shift(T)\in \bbZ$ of $T$ by 
	\[\shift(T)=-\langle \wt(T),\rho^\vee\rangle+\sum_{\alpha \in \Phi_+} \phi_\alpha(T)=\langle \wt(T),\rho^\vee\rangle+\sum_{\alpha \in \Phi_+} \eps_\alpha(T).\]
\end{definition}

The goal of this section is to show that the atomic number is constant along LL atoms. The atomic number plays a direct role in the definition of the charge.
Moreover, in \Cref{KLinN}, we show that the atomic numbers can be used to express Kazhdan--Lusztig basis $\{\undH_\lambda\}$ of the Hecke algebra in terms of the \emph{atomic basis} $\{\bfN_\mu\}_{\mu \in X}$ of \cite[(2.4)]{LPPPre}.

\begin{prop}\label{atomicthm}
	The atomic number $Z$ is constant along LL atoms.
\end{prop}
\begin{proof}
	By \Cref{ccareLL}, we only need to show that $Z$ is constant on $W$-orbits and $\alpha_n$-strings. 
	
	Let $T\in \calB(\lambda)$. We divide the proof in two parts. In Part \ref{shiftW} we show that, for any $i$, we have $\shift(s_i(T))=\shift(T)$ and in Part \ref{shiftfn} we show that $\shift(f_n(T))=\shift(T)$.

\begin{enumerate}
	\item \label{shiftW} \emph{We have $\shift(s_i(T))=\shift(T)$ for any $i$.}

	We need to show that for any $i$ we have $\shift(s_i(T))=\shift(T)$. Notice that $s_i(\alpha_{j,k})= \alpha_{j,k}$ unless $i\in \{j-1,j,k,k+1\}$. Let $\phi_{j,k}:=\phi_{\alpha_{j,k}}$.	
	We have
	\begin{align}\nonumber \shift(T)&-\shift(s_i(T))= \langle s_i(\wt(T))-\wt(T),\rho^\vee\rangle +\sum_{\alpha\in \Phi_+}\left(\phi_\alpha(T)-\phi_\alpha(s_i(T))\right)\\
	=&\langle s_i(\wt(T))-\wt(T),\rho^\vee\rangle+\phi_i(T)-\phi_i(s_i(T))\label{row1}\\
	&+ \sum_{\substack{\alpha\in \Phi_+\\ s_i(\alpha)=\alpha}}\left(\phi_{\alpha}(T)-\phi_{\alpha}(s_i(T))\right)\label{row2}\\
	&+ \sum_{k\geq i+1}\left(\phi_{i,k}(T)-\phi_{i,k}(s_i(T))\right)+ \sum_{k\geq i+1}\left(\phi_{i+1,k}(T)-\phi_{i+1,k}(s_i(T))\right)\label{row3}
	\\
	&+ \sum_{j\leq i-1}\left(\phi_{j,i}(T)-\phi_{j,i}(s_i(T))\right)+ \sum_{j\leq i-1}\left(\phi_{j,i-1}(T)-\phi_{j,i-1}(s_i(T))\right).\label{row4}
	\end{align}
	We show that each of the rows \eqref{row1}-\eqref{row4} in the above equation separately vanishes.
	\begin{enumerate}[leftmargin=2cm]
		\item[\textbf{Row \eqref{row1}.}] We have $\langle s_i(\wt(T))-\wt(T),\rho^\vee\rangle=-\langle \wt(T), \alpha_i^\vee\rangle$. We have $\phi_i(s_i(T))=\eps_i(T)$ and, by \eqref{phi-eps}, also
		$\phi_i(T)-\eps_i(T)=\langle \wt(T), \alpha_i^\vee\rangle$.
		\item[\textbf{Row \eqref{row2}.}] We claim that $\phi_\alpha(s_i(T))=\phi_\alpha(T)$ for any $\alpha\in \Phi_+$ such that $s_i(\alpha)=\alpha$.
		
		Assume $\alpha=\alpha_{j,k}$ and let $w:=s_j s_{j+1} \ldots s_{k-1}$ so that $\phi_\alpha=\phi_k\circ w^{-1}$.
		If $i\leq j-2$ or $i\geq k+2$ the claim is clear since $s_i$ commutes with both $f_k$ and $w^{-1}$.
		
		Assume now that $j<i<k$. Then we have $w^{-1}s_i=s_{i-1}w^{-1}$ and since $s_{i-1}$ commutes with $f_k$ we obtain \[\phi_\alpha(s_i(T))=\phi_k(w^{-1}s_i(T))=\phi_k(w^{-1}(T))=\phi_\alpha(T).\]

		\item[\textbf{Row \eqref{row3}.}] 
		This follows from $\phi_{i,k}(s_i(T))=\phi_k(s_{k-1}\ldots s_{i+1}(T))=\phi_{i+1,k}(T)$ and \break $\phi_{i+1,k}(s_i(T))=\phi_k(s_{k-1}\ldots s_i(T))=\phi_{i,k}(T)$.

		\item[\textbf{Row \eqref{row4}.}]
		Fix $j\leq i-1$. Let $w = s_j s_{j+1}\ldots s_{i-1}$ and $T':=w^{-1}(T)$. 
		We have \[\phi_{j,i}(T)=\phi_i(T'), \qquad \phi_{j,i-1}(T)=\phi_{i-1}(s_{i-1}w^{-1}(T))=\eps_{i-1}(T'),\]
		\[\phi_{j,i}(s_i(T))=\phi_i(s_{i-1}s_{i-2}\ldots s_j s_i(T))=\phi_i(s_{i-1} s_i s_{i-1}(T')),\]
		\[\phi_{j,i-1}(s_i(T))=\eps_{i-1}(s_{i-1}s_i s_{i-2} \ldots s_j(T))=\eps_{i-1}(s_{i-1}s_is_{i-1} (T')).\]
		
		The vanishing of Row \eqref{row4} follows from the following claim. For any $T'\in \calB(\lambda)$ we have
		\[\phi_i(T')+\eps_{i-1}(T')=\phi_i(s_{i-1}s_is_{i-1} T')+\eps_{i-1}(s_{i-1}s_is_{i-1}T').\]	
		This is a consequence of \Cref{21explicit}. In fact, $T'$ and $s_{i-1}s_is_{i-1}(T')$ belong to the same $\alpha_{i-1}+\alpha_i$-string.
	\end{enumerate}

\item	\emph{We have $\shift(T)=\shift(f_n(T))$.}\label{shiftfn}

	The proof is similar to Part \ref{shiftW}. We divide all the positive roots into $s_n$-orbits and consider them separately.
	\begin{itemize}
		\item If $\beta$ is a root such that $s_n(\beta)=\beta$, then
		\begin{equation}
		\label{sb=b} \phi_\beta(f_n(T))=\phi_\beta(T).
		\end{equation}
		\ In fact, we have $\beta=\alpha_{j,k}$ for some $k\leq n-2$ and $f_n$ commutes with $f_k$ and $s_{k-1}s_{k-2}\ldots s_j$.
		\item If $\beta\neq \alpha_n$ and $s_n(\beta)\neq \beta$ we claim that
		\begin{equation}\label{phibeta}
		\phi_\beta(T)+\phi_{s_n(\beta)}(T)=
		\phi_\beta(f_n(T))+\phi_{s_n(\beta)}(f_n(T)).
		\end{equation}
		We can assume $\beta=\alpha_{j,n}$ and $s_n(\beta)=\alpha_{j,n-1}$.
		Let $w=s_j\ldots s_{n-1}$ so that $\phi_\beta=\phi_n \circ w^{-1}$ and $\phi_{s_n(\beta)}=\phi_{n-1} \circ s_{n-1}w^{-1}=\eps_{n-1}\circ w^{-1}$.
		Let $T'=w^{-1}(T)$. We have
		\[ \phi_\beta(f_n(T))= \phi_n(w^{-1} f_n w(T')) = \phi_n(f_{w^{-1}(\alpha_n)}(T'))=\phi_n(f_{\alpha_{n-1}+\alpha_n}(T')).\]

		\Cref{phibeta} can be rewritten as
		\[\phi_n(T')+\eps_{n-1}(T')=\phi_n(f_{\alpha_{n-1}+\alpha_n}(T'))+\eps_{n-1}(f_{\alpha_{n-1}+\alpha_n}(T')). \]
		Finally, this follows directly from \Cref{21explicit}.
		\item Finally, if $\beta=\alpha_n$, then \begin{equation}\label{b=n}\phi_n(f_n(T))=\phi_n(T)-1\end{equation}
	\end{itemize}
	
	Notice that $\langle \wt(f_n(T)),\rho^\vee\rangle=\langle \wt(T),\rho^\vee\rangle-1$. The claim easily follows by combining \eqref{sb=b},\eqref{phibeta} and \eqref{b=n}. \qedhere
\end{enumerate}
\end{proof}

\section{Twisted Bruhat Graphs}\label{TBGsec}

In this section we study twisted Bruhat graphs in affine type $A$. Twisted Bruhat orders and graphs were first considered by Dyer in \cite{DyeHecke}. Our main goal is to prove \Cref{Gammam}, which is a technical result about the number of arrows in a portion of the twisted Bruhat graph given by all the elements smaller than a certain $\lambda\in X_+$.

\subsection{Biclosed Sets and Twisted Bruhat Orders}

Let $X^\vee$ denote the dual of $X$, i.e. the cocharacter lattice of $T$. Let $\Phi^\vee\subset X^\vee$ be its dual system (so both $\Phi$ and $\Phi^\vee$ are root systems of type $A_n$). For $\beta\in \Phi$, let $\beta^\vee$ denote the corresponding coroot in $\Phi^\vee$. We denote by $\Phi_+$ the set of positive roots and by $\Phi^\vee_+$ the set of positive coroots. Let $\alpha_i$ and $\alpha_i^\vee$ denote the simple root and coroots. For $i\leq j$, we also define $\alpha_{i,j}:=\alpha_i+\ldots +\alpha_j$ and $\alpha^\vee_{i,j}:=\alpha^\vee_i+\ldots +\alpha^\vee_j$. Let $\varpi_1,\ldots,\varpi_n$ denote the fundamental weights in $X$.

Let $\affW$ denote the affine Weyl group $\affW\cong W \ltimes \bbZ\Phi$ and let 
\[\affPhi^\vee:=\left\{ m\delta +\beta^\vee \mid \beta^\vee \in \Phi^\vee,\; m \in \bbZ\right\} \subset X^\vee \oplus \bbZ \delta\]
denote the set of real roots of the affine root system of $PGL_{n+1}(\bbC)$. The positive real roots are $\affPhi^\vee_+=\{m\delta+\beta^\vee \mid m>0\} \cup \Phi^\vee_+$ and are in bijection with the reflections in $\affW$. If $t\in \affW$ is a reflection we denote by $\alpha_t^\vee\in \affPhi^\vee_+$ the corresponding root. Then $\affW$ is a Coxeter group with simple reflections $s_0,s_1,\ldots,s_n$, where $s_0$ is the reflection corresponding to $\delta-\theta^\vee$, with $\theta^\vee$ the highest root in $\affPhi$ (see  \cite[\S 1]{ChargeGen} for more details).

\begin{definition}
	A subset $B\cu \affPhi^\vee_+$ of positive roots is said to be \emph{closed} if for any $\alpha_1,\alpha_2\in B$ and for any $k_1,k_2\in \bbR_{\geq 0}$ such that $k_1\alpha_1+k_2\alpha_2\in \affPhi^\vee_+$ we have $k_1\alpha_1+k_2\alpha_2\in B$. 
	
	We say that a set $B$ is \emph{biclosed} if both $B$ and $\affPhi^\vee_+\setminus B$ are closed.
\end{definition}

For $w\in \affW$ we denote by \[N(w):=\{\alpha \in \affPhi^\vee_+\mid w^{-1}(\alpha)\in \affPhi^\vee_-\}\] the set of inversions.
If $w=s_{i_1}\ldots s_{i_k}$ is a reduced expression for $w$ then \[N(w)=\{\alpha^\vee_{i_1},s_{i_1}(\alpha^\vee_{i_2}),\ldots,s_{i_1}s_{i_2}\ldots s_{i_{k-1}}(\alpha^\vee_{i_k})\}.\] 
By \cite[Lemma 4.1(d)]{DyeWeak}, every set $N(w)$ for $w\in \affW$ is biclosed and moreover, a biclosed set is finite if and only if it is of the form $N(w)$ for some $w\in \affW$.

\begin{definition}
	An infinite word $x:=s_{i_1}s_{i_2} \ldots s_{i_k}\ldots$ in $\affW$ is called \emph{reduced} if for any $j$ the expression $x_j:=s_{i_1}s_{i_2} \ldots s_{i_j}$ is reduced. 
	In this case we define $N(x):=\bigcup_{j=1}^{\infty}N(x_j)$.
\end{definition}

If $x$ is an infinite reduced word, then its set of inversions $N(x)$ is also biclosed.

\begin{definition}
	Given a biclosed set $B$ we can define the \emph{$B$-twisted length} of $w\in \affW$ as 
	\begin{equation}\label{elllambda}\ell_B(w):=\ell(w)-2|N(w)\cap B|+|B|.
	\end{equation}
	(We neglect the term $|B|$ if $|B|=\infty.$) 
	
	We define the \emph{$B$-twisted Bruhat order} $<_B$ as the order generated by $w<_B t w$ for $t\in \affW$ a reflection and $\ell_B (w) < \ell_B(tw)$. 
\end{definition}
We remark that this slightly differs from the Dyer's original definition in \cite[\S 1]{DyeHecke} where he defines the $B$-twisted length of an element $w$ as 
\[\ell^{\text{Dyer}}_B(w):=\ell(w)-2|N(w^{-1})\cap B|=\ell_B(w^{-1})-|B|.\]

\newcommand{\undw}{\underline{w}}
We consider the action of $\affW$ on $\affPhi^\vee_+$ defined as \[
\undw(\alpha^\vee)=\begin{cases} w(\alpha^\vee) &\text{if }w(\alpha^\vee)\in \affPhi^\vee_+\\
	-w(\alpha^\vee) &\text{if }w(\alpha^\vee)\in \affPhi^\vee_-.
\end{cases}\]
Equivalently, we have $\undw\alpha_t^\vee=\alpha_{wtw^{-1}}^\vee$ for any reflection $t\in \affW$. Observe that $\undw(N(w^{-1}))=N(w)$.

\begin{lemma}\label{ordersymdiff}
	 For any reflection $t\in T$ we have $tw\leq_B w\iff \alpha_t^\vee\in N(w)\ominus B$, where $\ominus $ denotes here the symmetric difference of sets. 
\end{lemma}

\begin{proof}
	This follows directly from \cite{DyeHecke} after the opportune translations. In fact, combining  \cite[Prop. 1.1]{DyeHecke} and  \cite[Prop. 1.2]{DyeHecke} we have that \begin{align*}\ell_B(tw)<\ell_B(w)&\iff \ell^{\text{Dyer}}_B(w^{-1}t)<\ell^{\text{Dyer}}_B(w^{-1})\\& \iff  \alpha_{w^{-1}tw}^\vee= \underline{w^{-1}}(\alpha_t^{\vee})\in N(w^{-1})\ominus  \underline{w^{-1}}( B)\\
&\iff \alpha_t^\vee \in \undw(N(w^{-1}))\ominus  B=N(w)\ominus B.		\qedhere
	\end{align*}

\end{proof}

If $B=\emptyset$ then $\ell_B$ and $\leq_B$ are the usual length and Bruhat order.
We are only interested in biclosed sets of the type $N(y)$ where $y$ is either an element of $W$ or an infinite reduced expression.

\begin{lemma}
	Let $y,v,w\in \affW$.
	\begin{enumerate}
		\item We have $\ell_{N(y)}(w)=\ell(y^{-1}w)$. 
		\item We have and $v\leq_{N(y)} w$ if and only if $y^{-1}v\leq y^{-1}w$ and $y$ is the unique minimal element with respect to $\leq_{N(y)}$. 
		\item  We have that 
		\begin{equation}\label{symmdiff}
			tw<_{N(y)} w \iff \alpha_t^\vee\in N(w)\ominus N(y)
		\end{equation}
	\end{enumerate}
\end{lemma}
\begin{proof}
	First notice that $N(y^{-1}w)=N(y^{-1})\ominus \underline{y^{-1}}(N(w))$.
We have \begin{align*}\ell_{N(y)}(w) &=|N(w)|-2|N(w)\cap N(y)|+|N(y)|=|N(w)\ominus N(y)|=|N(w)\ominus \underline{y}N(y^{-1})|=\\ &=
	|\underline{y^{-1}}N(w)\ominus N(y^{-1})|=|N(y^{-1}w)|=\ell(y^{-1}w).\end{align*}
	This shows (1). Moreover, the twisted order $<_{N(y)}$ is generated by $w<_{N(y)}tw$ for $y^{-1}w<y^{-1}tw$ so we have (2). Statement (3) follows directly from \Cref{ordersymdiff}.
\end{proof}

\begin{remark}
	The definitions given in this section generalize straightforwardly to the 
	 extended affine Weyl group $\extW=W\ltimes X$ (cf. \cite[\S 1.4]{ChargeGen}).
	We can define the set of inversions as $N(w)=\{ \alpha^\vee\in \affPhi^\vee_+ \mid s_\alpha w<w\}$, so that it makes sense for any $w\in \extW$. Then we obtain the definition of $B$-twisted length and $B$-twisted Bruhat order on $\extW$ by simply extending \Cref{elllambda}. 
\end{remark}

\subsection{Twisting the Bruhat graph along an infinite word}\label{twistedbruhatsection}

Recall that we assumed that $\Phi$ is of type $A_n$, so $\affW$ is the affine Weyl group of type $\tilde{A}_n$. Consider $\undc = s_0s_1s_2\ldots s_n$. Then $y_\infty:=\undc \undc \undc\ldots$ is an infinite reduced expression.
Let $y_m$ be the element given by the first $m$ simple reflections in $y_{\infty}$.

We order the roots in $N(y_\infty)$ as follows:
\begin{equation}\label{reflectionorder}
\delta-\alpha_{1,n}^\vee<\delta-\alpha_{2,n}^\vee<\ldots <\delta-\alpha_n^\vee<2\delta -\alpha_{1,n}^\vee<2\delta-\alpha_{2,n}^\vee<\ldots
\end{equation}
In this way, the roots in $N(y_m)$ are precisely the first $m$ roots in \eqref{reflectionorder}.

We just write $\ell_m$ and $\leq_m$ for $\ell_{N(y_m)}$ and $\leq_{N(y_m)}$, the $y_m$-twisted length and the $y_m$-twisted Bruhat order on $\affW$. 
Recall that $X\cong \extW/W$.
\begin{definition}
Let $m\in \bbZ_{\geq0} \cup \{ \infty\}$. We regard $\lambda\in X$ as the right coset in $\extW$ containing the translation $\texttt{t}_\lambda$. We denote by $\lambda_m\in \extW$ the element of minimal $y_m$-twisted length in the coset $\lambda$ and we set $\ell_m(\lambda):=\ell_m(\lambda_m)$.

We define the twisted  Bruhat order on $X$ by setting $\mu\leq_m \lambda$ if $\mu_m\leq_m \lambda_m$.
\end{definition}

If $m\in \bbZ_{\geq 0}$ we have \[\lambda_m= y_m(y_m^{-1}\lambda)_0\]
where $(y_m^{-1}\lambda)_0$ denotes the minimal element with respect to the usual length.
We write $\ell(\lambda)$ instead of $\ell_0(\lambda)$.

The Bruhat graph $\Gamma_X$ is the graph 
whose vertices are the weights $X$, and we have an arrow $\mu\ra \lambda$ if and only if there exists $\alpha^\vee\in \affPhi^\vee$ such that $s_{\alpha^\vee}(\mu)=\lambda$ and $\mu\leq \lambda$. The Bruhat graph can be obtained as the moment graph of the affine Grassmannian of $G^\vee$ (see \cite[\S 2.3]{ChargeGen} for details).

\begin{definition}
	For every $m\in \bbZ_{\geq 0} \cup \{ \infty\}$ we define $\Gamma_X^m$, the $y_m$\emph{-twisted Bruhat graph of} $X$, to be the directed graph whose vertices are the weights $X$ and where there is an edge $\mu\ra \lambda$ if there exists $\alpha^\vee\in \affPhi^\vee$ such that $s_{\alpha^\vee}(\mu)=\lambda$ and $\mu<_m \lambda$. 
	
	Similarly, we define $\Gamma_{\extW}^m$ to be the directed graph with $\extW$ as the set of vertices and where there is an edge $w\ra v$ if there exists $\alpha^\vee\in \affPhi^\vee$ such that $s_{\alpha^\vee}w=v$ and $w<_m v$. 
\end{definition}

Because of \eqref{symmdiff}, $\Gamma_{\extW}^m$ and $\Gamma_X^m$ can be obtained respectively from $\Gamma_{\extW}$ and $\Gamma_X$ by inverting the direction of all the arrows $x\ra s_{\alpha^\vee} x$ such that $\alpha^\vee\in N(y_m)$.

\begin{definition}
	For  $\mu\in X$, we denote by $\Arr_m(\mu)$ the number of arrows directed to $\mu$ in $\Gamma_X^m$. 
\end{definition}

For $\lambda\in X$, we know that the length of the minimal element in the corresponding coset is equal to the number of arrows in the moment graph $\Gamma_X$ pointing to $\lambda$. The same holds for the twisted graphs.

\begin{lemma}\label{Xorder}
Let $m\in \bbZ_{\geq 0}$, $x\in \extW$ and $\lambda \in X$.
	\begin{enumerate}
		\item The twisted length $\ell_m(x)$ is equal to the number of arrows directed to $x$ in  $\Gamma^m_{\extW}$.
		\item The twisted length $\ell_m(\lambda)$ is equal to $\Arr_m(\lambda)$, the number of arrows directed to $\lambda$ in $\Gamma_X^m$.
	\end{enumerate}	
\end{lemma}
\begin{proof}

	We begin with $x\in \extW$. For any reflection $r$, if $ry_m^{-1}x<y_m^{-1}x$, then $y_m r y_m^{-1}x<_m x$. Hence, conjugation by $y_m$ induces a bijection between the arrows pointing to $y_m^{-1}x$ in $\Gamma_{\extW}$ and the arrows pointing to $x$ in  $\Gamma_{\extW}^m$.

	Consider now $\lambda\in X$.
	We recall the following fact about reflections from \cite[Theorem 1.4]{BDSWNote}. If $r_1,r_2$ are reflections in a Coxeter group and the product $r_1r_2$ belongs to a reflection subgroup $W'$, then either $r_1,r_2\in W'$ or $r_1=r_2$.
	
	Let $\lambda_m\in \extW$ be the minimal element in the coset of $\lambda$ with respect to the length function $\ell_m$. 
	We have an edge $\mu\to \lambda$ in $\Gamma^m_X$ if and only if there exists $\alpha^\vee\in \affPhi^\vee$ such that $s_{\alpha^\vee}(\mu)=\lambda$ and $\ell_m(\mu_m)<\ell_m(\lambda_m)$.

	 Recall, by \cite[Remark 2.8]{ChargeGen}, that $s_{\alpha^\vee}(\mu)=\lambda$ is equivalent to  $s_{\alpha^\vee}\texttt{t}_\mu W=\texttt{t}_\lambda W$, where $\texttt{t}_{\lambda}\in \extW$ denotes the translation by $\lambda$. Notice that we have $\texttt{t}_{\lambda}W=\lambda_m W$, so $s_{\alpha^\vee}(\mu)=\lambda$ is equivalent to $s_{\alpha^\vee}\mu_mW=\lambda_m W$.

	From the first part, we know that there are $\ell_m(\lambda)$ reflections $r$ such that $r\lambda_m<_m \lambda_m$.The claim  now follows, since distinct reflections lead to different cosets. 
	Otherwise, if for two distinct reflections $r_1,r_2$ there exist $w_1,w_2\in W$ such that $r_1\lambda_m w_1=r_2\lambda_m w_2$, then $r_1r_2\in \lambda_m W\lambda_m^{-1}$, hence $r_1,r_2 \in \lambda_m W \lambda_m^{-1}$. It follows that $r_1\lambda_m, r_2\lambda_m$ are in the same coset of $\lambda_m$, but this is impossible since $\lambda_m$ is minimal.
\end{proof}

We write $s'_m$ for the $m$-th simple reflection occurring in $y_\infty$ (i.e., $s'_m=s_{m-1}$, with the index $m-1$ regarded $\mod{n+1}$) so we have $y_ms'_{m+1}=y_{m+1}$.
Consider the reflection \[t_{m+1}:=y_{m+1}y_m^{-1} =y_ms'_{m+1} y_m^{-1}.\] Notice that $\{\alpha^\vee_{t_{m+1}}\}= N(y_{m+1})\setminus N(y_m)$.
For every $x\in \extW$ we have
 
\begin{equation}\label{tm}
\ell_{m+1}(x)=\ell_m(t_{m+1}x)=\ell(s'_{m+1}y_m^{-1}x)=\ell_m(x)\pm 1.
\end{equation}

We abbreviate $t=t_{m+1}$.

\begin{lemma}\label{arrowsX}
	Assume $\mu<_m t\mu$, so that $\ell_m(t\mu)=\ell_m(\mu)+1$.
	Then conjugation by $t$ induces a bijection
	\[ \left\{\begin{array}{c}
	\text{ arrows directed to }\mu\\\text{ in }\Gamma_X^m
	\end{array}\right\}\xleftrightarrow{\sim}\left\{\begin{array}{c}
	\text{ arrows directed to }t\mu\\\text{ in }\Gamma_X^m
	\end{array}\right\} \setminus \{ \mu \ra t\mu\}.\]
	In particular, we have $\Arr_m(\mu)=\Arr_m(t\mu)-1$.
\end{lemma}
\begin{proof}
	
	By \Cref{Xorder} we know that the two sets have the same cardinality, so it suffices to show that conjugation by $t$ sends arrows directed to $\mu$ into arrows directed to $t\mu$.
	
	Let $r$ be a reflection such that
	$\ell_m(r \mu) <\ell_m(\mu)$. It is enough to show that $trtt\mu=tr\mu <_m t\mu$, or equivalently that $y_m^{-1}tr\mu <y_m^{-1}t\mu$.
	\begin{center}
		\begin{tikzpicture}
		\node (a) at (0.2,1.2) {$y_m^{-1}r\mu$};
		\node (b) at (2.2,1.2) {$y_m^{-1}\mu$};
		\node (c) at (-1,0) {$y_m^{-1}tr\mu=s'_{m+1}y_m^{-1}r\mu$};
		\node (d) at (3.2,0) {$s'_{m+1}y_m^{-1}\mu=y_m^{-1}t\mu$};
		\node at (1.25,1.2) {$<$};
		\node[rotate=-90] at (2.2,0.6) {$<$};
		\node at (1.25,0) {$<$};
		\end{tikzpicture}
	\end{center}
	where the top inequality and the inequality on the right are by hypothesis. The required inequality on the bottom follows from the Property Z of Coxeter groups \cite{DeoSome}. 
\end{proof}

We now want to restrict ourselves to the weights smaller than a given $\lambda\in X_+$.

\begin{definition}
	For every weight $\lambda\in X_+$ let $I(\lambda)=\{\mu \in X \mid \mu \leq \lambda\}$ denote the corresponding lower interval.
	We denote by $\Gamma^m_\lambda$ the restriction of $\Gamma^m_X$ to $I( \lambda)$. 
	
		For  $\mu\in X$ with $\mu\leq \lambda$ we denote by $\Arr_m(\mu,\lambda)$ the number of arrows directed to $\mu$ in $\Gamma_\lambda^m$. 
\end{definition}

Recall that $I(\lambda)$ can be characterized as the set of all weights which are in the same class of $\lambda \mod \bbZ \Phi$ and that are contained in the convex hull of $W\cdot \lambda$. 

In general $I(\lambda)$ is not a lower interval for the twisted order $<_m$.
Rather surprisingly, the analogue of \Cref{arrowsX} still holds for $\Gamma^m_\lambda$. However, on $\Gamma^m_\lambda$ conjugation by $t$ does not provide a bijection anymore, but we need to carefully modify it.

Recall that any two reflections in type $A_n$ generate a group which is either of type $A_1\times A_1$ or of type $A_2$.
We consider these two different possible cases in the next two Lemmas.

\begin{lemma}\label{a1a1}
	Let $t=t_{m+1}$ as before. 
	Let $\mu\in X$ be such that $\mu <t\mu \leq \lambda$. Let $r$ be a reflection commuting with $t$ and such that $r\mu<_m \mu$. Then $r\mu < \mu$ and $rt\mu < t\mu$.
\end{lemma}
\begin{proof}
	Since the reflections $r$ and $t$ commute, the roots $\alpha_t^\vee$ and $\alpha_r^\vee$ are orthogonal.
	
	By \Cref{arrowsX} we know that $tr\mu=rt\mu <_m t\mu$ in $X$. 
	Recall that $\alpha^\vee_t\in N(y_{m+1})$. There are no orthogonal roots in $N(y_{m+1})$ because \[(k_1\delta-\alpha^\vee_{j_1,n},k_2\delta-\alpha^\vee_{j_2,n})=(\alpha^\vee_{j_1,n},\alpha^\vee_{j_2,n})=\begin{cases}
		2& \text{if }j_1=j_2\\
		1& \text{if }j_1\neq j_2.
	\end{cases}\]. Hence $\alpha^\vee_r\not \in N(y_m)$.
	It follows that $r\mu < \mu$ and $rt\mu < t\mu$.
\end{proof}

Let $r\in \affW$ be a reflection. We denote by $\beta_r\in \Phi_+$ the positive root such that $\alpha_r^\vee=m\delta-\beta_r^\vee$ for some $m\in \bbZ$.

\begin{lemma}\label{reflections}
		Let $\mu\in X$ be such that $\mu <t\mu \leq \lambda$. Let $r$ be a reflection not commuting with $t$ (i.e., the reflections $r$ and $t$ generate a group of type $A_2$ and we have $rtr=trt$) and such that $r\mu<_m \mu$.
	\begin{enumerate}
		\item Assume $r\mu \leq \lambda$ and 
		$tr\mu\not \leq \lambda$. Then  $rt\mu\leq \lambda$ and $trt\mu\not\leq \lambda$.
		\item Assume $r\mu\not \leq \lambda$ and 
		$tr\mu \leq \lambda$. 
		Then  $rt\mu\not\leq \lambda$ and $trt\mu\leq \lambda$.
	\end{enumerate}
Moreover, in both cases we have $trt\mu<_m \mu$ and $rt\mu<_m t\mu$.
\end{lemma}
\begin{proof}

	Consider the polytope $P_\lambda :=\Conv(W\cdot \lambda)\cu X_\bbR=X\otimes_\bbZ \bbR$.
	Every $(n-1)$-dimensional facet of $P_\lambda$ is parallel to a hyperplane of the form $w(H_i)$ for some $w\in W$ and $1\leq i\leq n$, where
	\[H_i:=H_{\varpi^\vee_i}=\{\nu \in X_\bbR \mid \langle \nu,\varpi^\vee_i\rangle =0 \}. \]
	In other words, $H_i$ is equal to the span of all the simple roots $\alpha_j$, for $j\neq i$.

	Assume we are in the first case, that is $r\mu \leq \lambda$ and 
	$tr\mu\not \leq \lambda$.
	Hence, there exists a hyperplane $H$ supporting a $(n-1)$-dimensional facet of $P_\lambda$ which separates $tr\mu$ from $t\mu$ and $r\mu$. Assume that $H$ is parallel to $w(H_i)$ with $w\in W$ and $1\leq i\leq n$.
	
	We have \[tr\mu = t\mu + k \beta_{rtr},\] hence $\beta_{rtr} \not \in w(H_i)$. Similarly, the difference $tr\mu-r\mu$ is a multiple of $\beta_t$. Therefore, also $\beta_t\not \in w(H_i)$. 
	
	\begin{claim}
		We have $\beta_r\in w(H_i)$.
	\end{claim}
	\begin{proof}[Proof of the claim.]
		The reflections $r,t,rtr$ are the three reflections in the dihedral group $\langle r,t\rangle$, so one of $\alpha^\vee_r,\alpha^\vee_t$ and $\alpha^\vee_{rtr}$ is the sum of the other two. This means that for some choice of the signs we have $\alpha^\vee_r=\pm\alpha^\vee_t \pm \alpha^\vee_{rtr}$. Then, we also have \[\beta_r=\pm\beta_t\pm \beta_{rtr}\]
		(here we use that $\Phi$ is simply-laced so that $(\beta_1+\beta_2)^\vee= \beta_1^\vee+\beta_2^\vee$). After applying $w^{-1}$ we obtain
		\[
		w^{-1}( \beta_r)=\pm w^{-1}(\beta_t) \pm w^{-1}(\beta_{rtr}).
		\]
		
		Recall that $w^{-1}\beta_{rtr},w^{-1}\beta_t\not \in H_i$.
		Since we are in type $A$, this implies
		\[\langle w^{-1}(\beta_r),\varpi_i^\vee\rangle = \pm 1, 
		\qquad \langle w^{-1}(\beta_{rtr}),\varpi_i^\vee\rangle = \pm 1, \qquad 
		\langle w^{-1}(\beta_r),\varpi_i^\vee\rangle = \pm 1 \pm 1 \in \{-2,0,2\}.\]
		Since $w^{-1}(\beta_r)$ is a root, the only possibility is $\langle w^{-1}(\beta_r),\varpi_i^\vee\rangle=0$ and $\beta_r\in w(H_i)$. The claim is proved.\footnote{It is crucial that we are in type $A$ as the claim does not hold for other root systems, even of simply-laced type. If in $D_4$ we label by $1$ the simple root corresponding to the central vertex of the Dynkin diagram, the roots $\alpha_1+\alpha_2$, $\alpha_1+\alpha_3+\alpha_4$, $2\alpha_1 +\alpha_2 +\alpha_3 +\alpha_4$ are roots of a dihedral subgroup, but none of them lies in $H_1$.}
	\end{proof}
	
	We want to show that the graph $\Gamma_X^m$ looks like in \Cref{fig1}. The black arrows are given by hypothesis.
	
	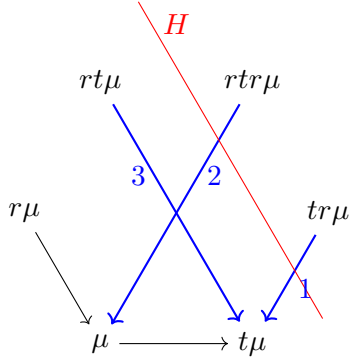
\begin{figure}[ht]
		\begin{center}
			\begin{tikzpicture}
			\node (x) at (240:2) {$\mu$};
			\node (tx) at (-60:2) {$t\mu$};
			\node (rtr) at (60:2) {$rtr\mu$};
			\node (rtx) at (120:2) {$rt\mu$};
			\node (rx) at (180:2) {$r\mu$};
			\node (trx) at (0:2) {$tr\mu$};
			\draw[->] (x) -- (tx);
			\draw[->] (rx) -- (x);
			
			\draw[->,blue, thick] (rtr) -- (x);	
			\draw[->,blue,thick] (trx) -- (tx);
			\draw[->,blue,thick] (rtx) -- (tx);
			\begin{scope}[xshift=-0.5cm]
			\draw[red] (90:2.8) -- (-30:2.8);
			\node[red] (H) at (0.5,2.5) {$H$};
			\end{scope}
			\node[blue] at (1.7,-1) {$1$};
			\node[blue] at (0.5,0.5) {$2$};
			\node[blue] at (-0.5,0.5) {$3$};
			\end{tikzpicture}
			\caption{A portion of the graph $\Gamma_X^m$ in the first case.}\label{fig1}
		\end{center}
	\end{figure}
	The weights $tr\mu$ and $rtr\mu$ differ only by a multiple of $\beta_r$ and, since $\beta_r\in w(H_i)$, they lie on the same side with respect to the hyperplane $H$. In particular, since $tr\mu \not \in P_\lambda$ then $rtr\mu\not \in P_\lambda$ as well.

	Since $tr\mu \not \leq\lambda$ and $t\mu\leq \lambda$, we have $tr\mu> t\mu$. On the other hand, we have $r\mu<_m \mu$ and by \Cref{arrowsX} we obtain $tr\mu <_m t\mu$ (arrow $1$ in \Cref{fig1}). This implies that we must have $\alpha^\vee_{rtr}\in N(y_m)$.

	Recall that $\{\alpha^\vee_t\}= N(y_{m+1})\setminus N(y_m)$.
	Therefore, both $\alpha^\vee_t$ and $\alpha^\vee_{rtr}$ lie in $N(y_{m+1})$.
	For all the roots $\alpha^\vee \in N(y_{m+1})$ we have $\langle \varpi_n,\alpha^\vee\rangle = -1$. Since $\alpha^\vee_r=\pm\alpha^\vee_t \pm \alpha^\vee_{rtr}$, we deduce $\alpha^\vee_r\not \in N(y_{m+1})$ and, in particular, $\alpha_r^\vee \not \in N(y_m)$.

	Since $trt\mu \not \leq \lambda$ and $\mu\leq \lambda$ we have $\mu <trt\mu$. Because $\alpha_{rtr}\in N(y_m)$ this is equivalent to $trt\mu <_m \mu$ (arrow 2). 
	
	Moreover, by \Cref{arrowsX}, $trt\mu <_m \mu$ implies $rt\mu <_m t\mu$ (arrow 3). Finally, since $\alpha_r \not \in N(y_m)$, this is in turn equivalent to $rt\mu < t\mu$, hence $rt\mu < \lambda$. The proof of the first part is complete.



	The second part of the second case is similar, but we report it anyway for completeness. In this case, we claim that the graph $\Gamma_X^m$ looks as in \Cref{fig2}.
	\begin{figure}[h]
		\begin{center}
			\begin{tikzpicture}
			\node (x) at (240:2) {$\mu$};
			\node (tx) at (-60:2) {$t\mu$};
			\node (rtr) at (60:2) {$rtr\mu$};
			\node (rtx) at (120:2) {$rt\mu$};
			\node (rx) at (180:2) {$r\mu$};
			\node (trx) at (0:2) {$tr\mu$};
			\draw[->] (x) -- (tx);
			\draw[->] (rx) -- (x);
			\draw[->,blue,thick] (rtx) -- (tx);
			\draw[->,blue,thick] (rtr) -- (x);	
			\draw[->,blue, thick] (rx) -- (trx);
			\begin{scope}[xscale=-1,xshift=-0.5cm]
			\draw[red] (90:2.8) -- (-30:2.8);
			\node[red] (H) at (0.5,2.5) {$H$};
			\end{scope}
			\end{tikzpicture}
		\end{center}
		\caption{A portion of the graph $\Gamma_X^m$ in the second case.}\label{fig2}	
	\end{figure}
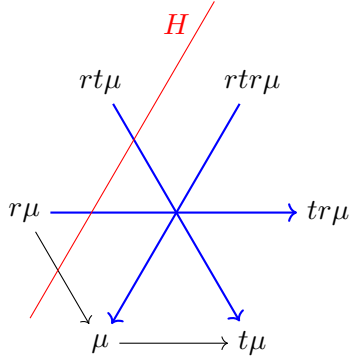

	Assume that $r \mu \not \in P_\lambda$ and that $tr\mu \in P_\lambda$. So there is a hyperplane $H$ supporting a $(n-1)$-dimensional facet of $P_\lambda$ dividing $r\mu$ from $\mu$ and $tr\mu$. Assume that $H$ is parallel to the hyperplane $H=w(H_i)$, with $w\in W$ and $1\leq 1\leq n$. 
	
	As in the first part, we have $\beta_r,\beta_t\not \in w(H_i)$, so $\beta_{rtr}$ must lie in $w(H_i)$. Moreover, we have $\alpha^\vee_r \in N(y_m)$, hence $\alpha^\vee_{rtr}\not \in N(y_m)$. The element $rt\mu$ is on the same side of $r\mu$ with respect to $H$, so also $rt\mu \not \in P_\lambda$.
	We deduce that $t\mu < rt\mu$, hence $rt\mu <_m t\mu$ which is equivalent to $rtr\mu <_m \mu$ by \Cref{arrowsX}. 
	
	Finally, since $\alpha_{rtr}^\vee\not \in N(y_m)$, we also deduce that $rtr\mu<\mu<\lambda$ as required.
\end{proof}

Motivated by \Cref{reflections}, we give the following definition.

\begin{definition}
	Let $\mu\in X$ and $\lambda\in X_+$ be such that $\mu<t\mu\leq \lambda$. Let $r$ be a reflection such that $r\mu <_m \mu$. 
	We say that $r$ is \emph{$t$-reversing} for $\mu$ if one of the following conditions holds:
	\begin{itemize}
		\item $r\mu \leq \lambda$ and $tr\mu \not\leq \lambda$,
		\item $r\mu \not \leq \lambda$ and $tr\mu \leq \lambda$.
	\end{itemize}
\end{definition}
Notice that  $r$ can be $t$-reversing for $\mu$ only if it does not commute with $t$. In fact, if $r$ and $t$ commute by \Cref{a1a1} we have $r\mu <\mu\leq \lambda$ and $rt\mu=tr\mu <t\mu\leq \lambda$. From \Cref{reflections} we see that $r$ is $t$-reversing for $\mu$ if and only if $trt$ is also $t$-reversing for $\mu$.

\begin{prop}\label{Gammam}
	Let $\mu\in X$ and $\lambda\in X_+$ be such that $\mu< t\mu\leq \lambda$. Then
	\[	\left\lvert\left\{\begin{array}{c}
	\text{ arrows directed to }\mu\\\text{ in }\Gamma_\lambda^m
	\end{array}\right\}\right\rvert=\left\lvert\left\{\begin{array}{c}
	\text{ arrows directed to }t\mu\\ \text{ in }\Gamma_\lambda^m
	\end{array}\right\}\right\rvert -1.\]
In particular, $\Arr_m(\mu,\lambda)=\Arr_\mu(t\mu,\lambda)-1$.
\end{prop}
\begin{remark}
	This result is specific about type $A$ and its analogue in other types seems to fail in most cases. See for example $\lambda = 2\varpi_1 +2\varpi_2 $ in type $B_2$. 
\end{remark}
\begin{proof}
	To prove the claim, we want to define a bijection between arrows pointing to $\mu$ and arrows pointing to $t\mu$.
	This is done modifying the bijection given by conjugation by $t$ (cf. \Cref{arrowsX}) by taking into account the notion of $t$-reversing reflections.
	
	We begin by defining a map $\phi$ as follows.
	\[ \phi: \{r\mu\in X \mid r\mu <_m \mu\} \ra \{r't\mu\in X \mid r't\mu <_m t\mu\} \setminus \{ \mu\}\]
	\[\phi(r\mu) = \begin{cases}
	tr\mu & \text{if }r\text{ is not $t$-reversing for }\mu\\
	rt\mu & \text{if }r\text{ is $t$-reversing for } \mu\\
	\end{cases}\]
	
	Notice that if $r\mu <_m \mu$, then $tr\mu <_m t\mu$ by
	\Cref{arrowsX}. If $r$ is $t$-reversing we are in the setting of \Cref{reflections} and $rt\mu <_m t\mu$.
	So the map $\phi$ is well defined.
	
	By \Cref{arrowsX}, we already know that $\phi$ is a map between sets with the same cardinality.
	To prove that $\phi$ is bijective, it is sufficient to show that $\phi$ is injective.
	
	Suppose that $\phi(r'\mu)=\phi(r''\mu)$ with $r'\mu\neq r''\mu$. By \cite[Lemma 1.3]{ChargeGen} then we have $r'\neq r''$. Without loss of generality we can assume $\phi(r'\mu)=r't\mu$ and $\phi(r''\mu)=tr''\mu$ (i.e., we assume that $r'$ is $t$-reversing while $r''$ is not) so that $r't\mu = tr''\mu$. This implies $r''=tr't$, again by \cite[Lemma 1.3]{ChargeGen}. 
	On the other hand, by \Cref{reflections} if $r'$ is $t$-reversing, then so is $r''=tr't$. This leads to a contradiction, hence $\phi$ is injective.

	We now want to show that $\phi$ restricts to a bijection between arrows in $\Gamma^m_\lambda$. That is, we claim that $\phi$ restricts to a bijection as follows.
	\[ \phi: \{r\mu\in X \mid r\mu <_m \mu \text{ and }r\mu \leq \lambda\} \xra{\sim} \{r't\mu\in X \mid r't\mu <_m t\mu\text{ and }r't\mu\leq \lambda\} \setminus \{ \mu\}\]
	
	If $r\mu \leq \lambda$ then either $r$ is $t$-reversing, in which case $\phi(r\mu)= rt\mu \leq \lambda$ by \Cref{reflections}, or
	$r$ is not $t$-reversing, in which case $\phi(r\mu)=tr\mu \leq \lambda$. Similarly, 
	if $r\mu \not \leq \lambda$ then either $r$ is $t$-reversing, in which case $\phi(r\mu)= rt\mu \not \leq \lambda$, or
	$r$ is not $t$-reversing, in which case $\phi(r\mu)=tr\mu \not\leq \lambda$. Summing up, we have shown that $r\mu \leq \lambda$ if and only if $\phi(r\mu) \leq \lambda$.	
\end{proof}

\subsection{The twisted Bruhat graph at \texorpdfstring{$\infty$}{infinity}}

The Bruhat graph $\Gamma^{m+1}_\lambda$ can be obtained from $\Gamma^m_{\lambda}$ by reversing the direction of all the edges with label $\alpha_t^\vee$, where $t=t_{m+1}=y_{m+1}y_m^{-1}$. Notice that $\mu<t\mu \iff\mu <_m t\mu \iff \mu >_{m+1}t\mu$.
Therefore, we have
	\begin{equation}\label{Arrm+1}
		\Arr_{m+1}(\mu,\lambda)=\begin{cases}
			\Arr_{m}(\mu,\lambda)-1& \text{if }t\mu<\mu,\\
			\Arr_{m}(\mu,\lambda)+1& \text{if }\mu<t\mu\leq \lambda',\\
			\Arr_{m}(\mu,\lambda)& \text{if }\mu<t\mu\not \leq \lambda'\text{ or if }t\mu=\mu.
		\end{cases}
	\end{equation}

Recall that the Bruhat graph $\Gamma_\lambda^\infty$ can be obtained from $\Gamma_\lambda$ by inverting the orientation of all the edges with label in $N(y_\infty)$. Since $\Gamma_\lambda$ is finite, there exists $M\gg 0$ such that $\Gamma_\lambda^M=\Gamma_\lambda^\infty$.

Let $\mu,\nu <\lambda$ and suppose that $\nu <_{\infty}\mu$, so there is an edge connecting $\nu$ and $\mu$ in $\Gamma_{\lambda}^\infty$. This means that there exists $\alpha^\vee\in \affPhi^\vee_+$ with $s_{\alpha^\vee}(\mu)=\nu$. By \cite[Lemma 2.7]{ChargeGen} we have in this case that $\mu-\nu=k\beta$ with $\beta\in \Phi_+$ and $k\in \bbZ$, and $\alpha^\vee=(k-\langle\mu,\beta^\vee\rangle)\delta-\beta^\vee$. 

Let $\Phi_{n-1}\subset \Phi$ be the root subsystem with simple roots $\alpha_1,\ldots,\alpha_{n-1}$. Notice that if $\beta\in \Phi_{n-1}$, then $m\delta -\beta^\vee \not \in N(y_\infty)$ for any $m\in \bbZ$.
Thus the orientation of the edge  $\nu\to\mu$ in $\Gamma_\lambda^\infty$ is different from the orientation of the same edge in $\Gamma_\lambda$ if and only if $\beta\in \Phi_+\setminus \Phi_{n-1}$ and $k>\langle \mu,\beta^\vee\rangle$.

Recall from \cite[Eq. (9)]{ChargeGen} that
\[ \nu <\mu 
\iff \langle \mu, \beta^\vee\rangle \geq k >0\text{ or }\langle \mu,\beta^\vee\rangle <k<0.\]
For $\beta\in \Phi_+\setminus \Phi_{n-1}$, we have then that 
\[\nu <_{\infty} \mu\iff k>0.\]
Hence, as illustrated in \Cref{figbruhat}, if $\beta\in \Phi_+\setminus \Phi_{n-1}$, the restriction of $<_{\infty}$  to the subset $\{\mu+k\beta\}_{k\in \bbZ}$ coincides the dominance order on weights.

\begin{figure}
\begin{tikzpicture}
	\tikzstyle{every node}=[draw,circle,fill=black,minimum size=5pt, inner sep=0pt]
	\draw (-6,0) node (-6) {};
	\draw (-4,0) node (-4) {};
	\draw (-2,0) node (-2) {};
	\draw (0,0) node (0) {};
	\draw (2,0) node (2) {};
	\draw (4,0) node (4) {};

			\draw (6,0) node (6) {};
	\tikzstyle{every node}=[]

\path (-6) ++(0,0.4) node {$\mu-3\beta$};
\path (-4) ++(0,0.4) node {$\mu-2\beta$};
\path (-2) ++(0,0.4) node {$\mu-\beta$};
\path (0) ++(0,0.4) node {$\mu$};
\path (2) ++(0,0.4) node {$\mu+\beta$};
\path (4) ++(0,0.4) node {$\mu+2\beta$};
\path (6) ++(0,0.4) node {$\mu+3\beta$};
	\path[->,red,very thick] (2) edge[out=160,in=30] node[above] {\scriptsize$\delta-\alpha^\vee$} (-4);
	\path[->,red,very thick] (0) edge[out=-150, in=-40] node[below] {\scriptsize$2\delta-\alpha^\vee$} (-4);
	\path[->,red,very thick] (-2) edge node[below] {\scriptsize$3\delta-\alpha^\vee$} (-4);
	\path[->,blue,dashed,very thick] (-6) edge node[below] {\scriptsize$5\delta-\alpha^\vee$} (-4);

\end{tikzpicture}\\
\begin{tikzpicture}
	\tikzstyle{every node}=[draw,circle,fill=black,minimum size=5pt, inner sep=0pt]
	\draw (-6,0) node (-6) {};
	\draw (-4,0) node (-4) {};
	\draw (-2,0) node (-2) {};
	\draw (0,0) node (0) {};
	\draw (2,0) node (2) {};
	\draw (4,0) node (4) {};
		\draw (6,0) node (6) {};
	\tikzstyle{every node}=[]
	\path (-6) ++(0,0.4) node {$\mu-3\beta$};
	\path (-4) ++(0,0.4) node {$\mu-2\beta$};
	\path (-2) ++(0,0.4) node {$\mu-\beta$};
	\path (0) ++(0,0.4) node {$\mu$};
	\path (2) ++(0,0.4) node {$\mu+\beta$};
	\path (4) ++(0,0.4) node {$\mu+2\beta$};
	\path (6) ++(0,0.4) node {$\mu+3\beta$};
	\path[->,red,very thick] (-2) edge[out=-30,in=-150] node[above] {\scriptsize$\alpha^\vee$} (2);
	\path[->,blue,dashed,very thick] (-2) edge[out=-30,in=-150]  (2);
	\path[->,red,very thick] (0) edge node[above] {\scriptsize$\delta+\alpha^\vee$} (2);
	\path[->,blue,dashed,very thick] (0) edge  (2);
	\path[->,blue,dashed,very thick] (-4) edge [out=30,in=150] node[above] {\scriptsize$\delta-\alpha^\vee$} (2);
	\path[->,blue,dashed,very thick] (-6) edge [out=-20,in =-160] node[above] {\scriptsize$5\delta-\alpha^\vee$} (2);
\end{tikzpicture}
\caption{In this example we have $\langle \mu,\beta^\vee\rangle=0$ and $\lambda=\mu+3\beta$. In red, the arrows pointing to $\mu-2\beta$ (above) and $\mu+\beta$ (below) in $\Gamma_\lambda$, in blue, some arrows pointing to $\mu-2\beta$ and $\mu+\beta$ in $\Gamma^\infty_\lambda$.} \label{figbruhat}
\end{figure}
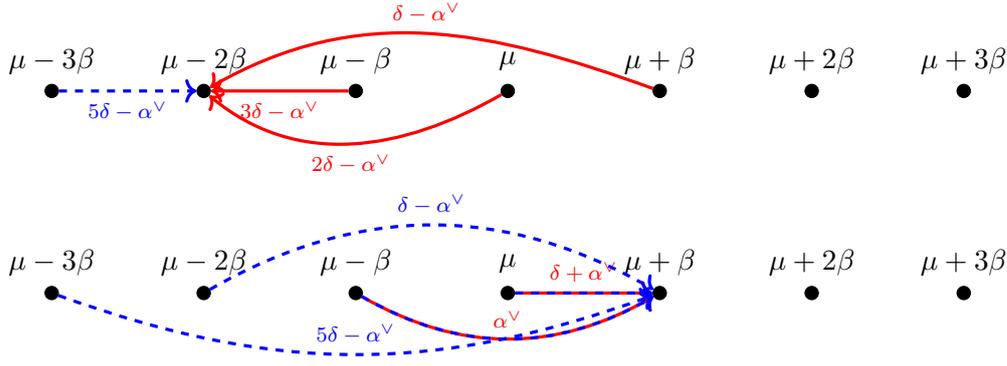

\begin{definition}
	For $\beta\in \Phi_+$ and $m\in \bbN \cup \{\infty\}$ we denote by $\Arr_m^\beta(\mu,\lambda)$ the number of arrows of the form $\mu+k\beta\to \mu$ in $\Gamma_\lambda^m$, with $k\in \bbZ$. 
	
	We also denote by $\ell_m^\beta(\mu)$ the number of arrows of the form $\mu+k\beta\to \mu$ in $\Gamma^m_X$.
	We write $\ell^\beta(\mu)$ instead of $\ell_0^\beta(\mu)$.
\end{definition}

By \cite[Lemma 2.7]{ChargeGen} we have $\sum_{\beta\in \Phi_+}\Arr^m(\mu,\lambda)$. Notice that $\Arr_0^\beta(\mu,\lambda)=\ell_0^\beta(\mu)=:\ell^\beta(\mu)$ for any $\mu \leq \lambda$. So we have $\sum_{\beta\in \Phi_+} \Arr_0^\beta(\mu,\lambda)=\sum_{\beta \in \Phi_+}\ell^\beta(\mu)=\ell(\mu)$. 

\begin{lemma}\label{inftyarrows}
	Let $\lambda\in X_+$ and $\mu \leq \lambda$. We have
	\[ \Arr_{\infty}(\mu,\lambda)=\sum_{\beta\in \Phi_+\setminus \Phi_{n-1}} \max\{k\in\bbN  \mid \mu-k\beta \leq \lambda\}+\sum_{\beta\in (\Phi_{n-1})_+}\ell^\beta(\mu)\]
\end{lemma}
\begin{proof}
	We need to count the number of roots $\alpha^\vee \in \affPhi^\vee_+$ such that $s_{\alpha^\vee}(\mu)<_\infty \mu$ and $s_{\alpha^\vee}(
	\mu)\leq \lambda$. 
	This is equivalent to sum $\Arr_\infty^\beta(\mu,\beta)$ over all $\beta\in \Phi_+$, where
	\[ \Arr^\beta_\infty(\mu,\lambda):=\left\lvert\{k\in \bbZ \mid 
	\mu-k\beta<_\infty \mu\text{ and }\mu-k\beta\leq \lambda\}\right\rvert.\] 	

For $\beta\in \Phi_+\setminus \Phi_{n-1}$, by the discussion above, we have  $\mu-k\beta<_\infty \mu$ if and only if $k>0$. 	
If $\beta\in (\Phi_{n-1})_+$, then $\mu-k\beta<_\infty \mu$ if and only if $\mu-k\beta< \mu$. In this case, we have $\Arr_\infty^\beta(\mu,\lambda)=\ell^\beta(\mu)$.
\end{proof}
\section{Swapping functions and charge statistics}

For $T\in \calB(\lambda)$ we denote by $\wt(T)\in X$ the weight of $T$.
Let $\calB(\lambda)_\mu$ denote the set of elements of the crystal of weight $\mu$. Let $\calB^+(\lambda)=\{ T\in \calB(\lambda) \mid \wt(T)\in X_+\}$.
Recall the definition of charge from \cite{LSLe,LLTCrystal}.
\begin{definition}
	A \emph{charge statistic} on $\calB(\lambda)$ is a function $c:\calB^+(\lambda)\ra \bbZ_{\geq 0}$ such that for any $\mu\in X_+$ with $\mu\leq \lambda$ we have
	\begin{equation}\label{chargeeq}h_{\mu,\lambda}(q^{\frac12})=K_{\lambda,\mu}(q)=\sum_{T\in \calB(\lambda)_\mu} q^{c(T)}\in \bbZ[q].
	\end{equation}	
		
\end{definition}

\subsection{Affine Grassmannian and affine root system}

Let $G=SL_{n+1}(\bbC)$ and let $G^\vee=PGL_{n+1}(\bbC)$ be the Langlands dual group.
Let $\Gr^\vee=G^\vee((t))/G^\vee[[t]]$ denote the affine Grassmannian of $G^\vee$. There is an action of the \emph{augmented torus} $\affT=T^\vee\times \bbC^*$ on $\Gr^\vee$, where $T^\vee\cu G^\vee$ is a maximal torus of $G^\vee$ consisting of diagonal matrices and $\bbC^*$ acts by loop rotations.
The weight lattice $X$ can be identified with the character lattice of $T$ and with the cocharacter lattice of $T^\vee$. The fixed points of the $\affT$-action on $\Gr^\vee$ are $\{L_\mu\}_{\mu \in X}$, where $L_\mu$ is the projection of the point in $T^\vee[t,t^{-1}]\subset T^\vee((t))$ defined by the morphism $\mu:\bbC^*\to T^\vee$.

We denote by $\affX:=X\oplus \bbZ d$ the cocharacter lattice of $\affT$.
Every cocharacter $\eta\in \affX$ determines a one-parameter subgroup $\bbC^*\cu \affT$ acting on the affine Grassmannian. For any weight $\mu$, we denote by $\HL^\eta_\mu$ the corresponding hyperbolic localization functor (cf. \cite[\S 2.4]{ChargeGen}).

We denote by $\IC_\lambda$ the intersection cohomology sheaf with complex coefficients supported on the Schubert variety $\sch{\lambda}:=\bar{G^\vee[[t]]\cdot L_\lambda}\subset \Gr^\vee$.
For $\eta \in \affX$, we study the graded dimension of the hyperbolic localization of $\IC_\lambda$. 
We call 
\[\htil^\eta_{\mu,\lambda}(v):=\grdim \HL^\eta_\mu(\IC_\lambda)\]
the \emph{renormalized $\eta$-Kazhdan--Lusztig polynomial}.


 Each real root $\alpha^\vee\in \affPhi^\vee$ defines a hyperplane
\[ H_{\alpha^\vee}=\{\eta \in \affX \otimes \bbR \mid \langle \eta ,\alpha^\vee \rangle =0\}.\]

Let us now examine how $\HL^\eta$ depends on $\eta$.
We consider the complement of all the walls $H_{\alpha^\vee}$ in $\affX\otimes \bbR$, and call \emph{chamber} its connected components. Then the hyperbolic localization $\HL^\eta$ does not change for $\eta$ within to the same chamber \cite[Proposition 2.29]{ChargeGen}.

If we are interested in studying $\HL_\mu^\eta(\IC_\lambda)$ for a fixed $\lambda$, then we can restricts ourselves to consider the walls $H_{\alpha^\vee}$ for $\alpha^\vee \in \affPhi^\vee(\lambda)$, where \[ \affPhi^\vee(\lambda):=\{\affPhi^\vee\mid \exists \mu\leq \lambda \text{ with }s_{\alpha^\vee}(\mu)\leq \lambda\}\subset \affPhi^\vee.\] 

The set $\affPhi^\vee(\lambda)$ is finite for every $\lambda$, and we refer to the walls $H_{\alpha^\vee}$ for $\alpha^\vee\in \affPhi^\vee(\lambda)$ as \emph{$\lambda$-walls}. We refer to the connected components of the $\lambda$-walls in $\affX\otimes \bbR$ as \emph{$\lambda$-chambers}. We say that two $\lambda$-walls are \emph{adjacent} if they are separated by a single $\lambda$-wall.

When $\eta$ is in the KL region (cf. \cite[Definition 2.17]{ChargeGen}), i.e. if $\langle \eta,\alpha^\vee\rangle >0$ for all positive affine real roots, then $\eta$-Kazhdan--Lusztig polynomials $\htil^\eta_{\mu,\lambda}(v)$ coincide with the ordinary Kazhdan--Lusztig polynomials.

We extend the definition of charge to renormalized $\eta$-Kazhdan--Lusztig polynomials. 

\begin{definition}
	Let $\eta\in \affX$ be a regular cocharacter.
	A \emph{renormalized charge statistic} (or \emph{recharge} for short) for $\eta$ on $\calB(\lambda)$ is a function $r(\eta,-):\calB(\lambda) \ra \frac12\bbZ$
	such that 
	\[\htil^\eta_{\mu,\lambda}(q^{\frac12})=\sum_{T\in \calB(\lambda)_\mu} q^{r(\eta,T)}\in \bbZ[q^{\frac12},q^{-\frac12}].\]
\end{definition}

Therefore, if $\eta$ is in the KL region and $r(\eta,-)$ is a recharge for $\eta$, then we can recover a charge statistic $c$ as in \eqref{chargeeq} by setting \begin{equation}\label{charge}
	c(T):=r(\eta,T)+\frac12 \ell(\wt(T)).
\end{equation}

Another important set of cocharacters are those in the MV region \cite[Definition 2.18]{ChargeGen}, the subset of dominant cocharacters in $X\subset \affX$. If $\eta_{MV}$ lies in the MV region, then $\HL^{\eta_{MV}}_\mu(\IC_\lambda)$ is concentrated in a single cohomological degree, so there is only one recharge statistic, given by
\begin{equation*}\label{rMV}
	r(\eta_{MV},T)=-\langle \wt(T),\rho^\vee\rangle.
\end{equation*}

We recall the definition of swapping functions from \cite[Definition 2]{ChargeGen}
\begin{definition}
	Let $\alpha^\vee$ be a positive real root of  the affine root system of $G^\vee$, let $\eta$ be a cocharacter and $r(\eta,-)$ a recharge for $\eta$. Then a \emph{swapping function} $\psi=\{\psi_\nu\}_{\nu\in X, \nu>s_{\alpha^\vee}(\nu)}$ for $r(\eta,-)$ and $\alpha$ is
	a collection of injective functions $\psi_{\nu}:\calB(\lambda)_{\nu} \ra \calB(\lambda)_{s_{\alpha^\vee}(\nu)}$ such that
	\begin{equation}\label{psiintro}
		r(\eta,\psi_\nu(T))=r(\eta,T)-1 \qquad \text{ for every }T\in \calB(\lambda)_\nu. 
	\end{equation}	
\end{definition}

Let $\eta_1$ and $\eta_2$ be two cocharacters lying on opposite sides of a wall $H_{\alpha^\vee}$ with $\langle \eta_1,\alpha^\vee\rangle <0<\langle \eta_2,\alpha^\vee\rangle$. Suppose we have a recharge $r(\eta_1,-)$ for $\eta_1$. Then, there exists a swapping function $\psi$ for $r(\eta_1,-)$ (
\cite[Lemma 3.4]{ChargeGen}). We can use the swapping function to construct a recharge for $\eta_2$ on the other side of the wall $H_{\alpha^\vee}$ as follows.

\begin{thm}\label{mainthmgen}
	Let $r(\eta_2,-)$ be the function on $\calB(\lambda)$ obtained by swapping  the values of $r(\eta_1,-)$ as indicated by a swapping function $\psi$, i.e. we have
	\begin{equation}\label{reta2}
		r(\eta_2,T)=\begin{cases}r(\eta_1,T)-1& \text{if }s_{\alpha^\vee}(\wt(T))<\wt(T)\\
			r(\eta_1,T)+1& \text{if }\wt(T)<s_{\alpha^\vee}(\wt(T))\leq \lambda\text{ and }T\in \Ima(\psi)\\
			r(\eta_1,T)& \text{if }\wt(T)\leq s_{\alpha^\vee}(\wt(T))\text{ and }T\not\in \Ima(\psi).
		\end{cases}
	\end{equation}
	The $r(\eta_2,-)$ is a recharge for $\eta_2$.
\end{thm}

In the situation of \Cref{mainthmgen}, we say that $\psi$ is a swapping function between $\eta_1$ and $\eta_2$.

Assume we have a sequence of cocharacters $\eta_0,\ldots,\eta_M$ with $\eta_0$ in the MV region, $\eta_M$ in the KL region and such that $\eta_i$ and $\eta_{i+1}$ lie in adjacent $\lambda$-chambers for any $0\leq i<M$. Let $\gamma_i^\vee$ be the root of the wall dividing $\eta_i$ and $\eta_{i+1}$.
Starting with $r(\eta_0,-)$ we can inductively construct a recharge for any $\eta_{i+1}$ and a swapping function for $\eta_i$ and $\gamma_i^\vee$. At the end of the process we obtain a recharge in the KL chamber, hence a charge. 
We say that a recharge obtained through this process is a recharge \emph{obtained via swapping operations}.

Even in type $A$, we do not know how to concretely construct swapping functions $\psi$ in general.
In what follows, we will restrict ourselves to the case of a specific sequence of cocharacters where we can explicitly construct the swapping functions in terms of the modified crystal operators.

\subsection{Recharge in the KL Region}\label{KLrecharge}

We now state the main result of this paper. We fix $\lambda\in X_+$ and we consider the crystal $\calB(\lambda)$.

\begin{thm}\label{main}
	Let $\eta$ be in the KL chamber. Then
	\begin{equation}\label{rfinal}
		r(\eta,T)=\shift(T)-\ell(\wt(T))
	\end{equation}
	is a recharge for $\eta$. Moreover, the recharge above can be obtained via swapping operations.
\end{thm}

The proof of \Cref{main} is done by induction on the rank of $\Phi$ and it will be completed in \Cref{sec:conclusions}. The statement is trivial if $\rk(\Phi)=0$. The case $\rk(\Phi)=1$ is discussed in detail in \cite[\S 3.1]{ChargeGen}.

\begin{definition}
	Let $\beta\in \Phi_+$. For $\mu\in X$, we define the \emph{length along $\beta$} of $\mu$ as 
	\[ \ell^\beta(\mu)= \begin{cases}
		\langle \mu,\beta^\vee\rangle &\text{if }\langle \mu,\beta^\vee\rangle\geq 0\\
		-\langle \mu,\beta^\vee\rangle-1&\text{if }	\langle \mu,\beta^\vee\rangle<0.
	\end{cases}\] 
\end{definition}
In rank 1, the length along $\beta$ is simply the length. In general, by \cite[Eq. (10)]{ChargeGen}, we have that \[\ell(\mu)=\sum_{\beta\in \Phi^+} \ell^\beta(\mu).\]
We can rewrite \eqref{rfinal} as
\begin{equation}\label{rfinal2} r(\eta,T)=-\langle \wt(T),\rho^\vee\rangle+ \sum_{\beta\in\Phi_+}\left(\phi_\beta(T)-\ell^\beta(\wt(T))\right).\end{equation}

\subsection{The parabolic region}

Let $n=\rk(\Phi)$.
For $k\leq n$ let $\Phi_k\cu \Phi$ be the sub-root system with simple roots $\alpha_1,\ldots,\alpha_k$ and let $\affPhi_k^\vee\cu \affPhi^\vee$ be the corresponding affine root system.
\begin{definition}
	Let $\lambda \in X_+$.
	We say that a regular cocharacter $\eta$ is in the \emph{$k$-th $\lambda$-parabolic region} (or simply $k$-th parabolic region, if $\lambda$ is clear from context) if
	\begin{itemize}
		\item $\langle \eta,\alpha^\vee\rangle>0$ for every $\alpha^\vee\in (\affPhi^\vee_k)_+$ and for every $\alpha^\vee\in \affPhi^\vee_+$ of the form $m\delta+\beta^\vee$ such that $m\geq 0$ and $\beta\in \Phi_+$,
		\item $\langle \eta,\alpha^\vee\rangle<0$ for every $\alpha^\vee\in \affPhi^\vee_+(\lambda)$ of the form $m\delta-\beta^\vee$ such that $m> 0$ and $\beta\in \Phi_+\setminus \Phi_k$.
	\end{itemize} 
\end{definition}

Let $X_{n-1}\cu X$ denote the subgroup generated by $\varpi_1,\varpi_2,\ldots,\varpi_{n-1}$.

\begin{lemma}\label{rparabolic}
	Let $n=\rk(\Phi)$	and assume \Cref{main} holds for $\Phi_{n-1} \cu \Phi$. Let $\eta_P$ be in the $(n-1)$-th parabolic region. Then
	\[r(\eta_P,T)=-\langle \wt(T),\rho^\vee\rangle+\sum_{\alpha\in(\Phi_{n-1})_+}\left(\phi_\alpha(T)-\ell^\alpha(\wt(T))\right)\]
	is a recharge for $\eta_P$.
\end{lemma}
\begin{proof}
	The walls that separate the MV region from the $(n-1)$-th parabolic region are all of the form $H_{m\delta-\beta^\vee}$ with $m> 0$ and $\beta \in (\Phi_{n-1})_+$. 
	Intersecting them with $\affX_{n-1}=(X_{n-1})_\bbR \oplus \bbR d$, we obtain precisely the walls separating the MV region from the KL region in type $A_{n-1}$.
	
	After Levi branching to $\Phi_{n-1}$, the crystal $\calB(\lambda)$ splits up as the disjoint union of crystals of type $A_{n-1}$. Let $\calB(\lambda')$ be one of these crystals.
	
	\Cref{main} holds for $\calB(\lambda')$ by assumption. Hence, we can find a sequence of cocharacters
	$\eta'_0,\ldots,\eta'_M\in X_{n-1}\otimes \bbZ d$ such that 
	$\eta'_0$ is in the MV chamber, $\eta'_M$ in the KL region  and that for any $i$ the cocharacters $\eta'_i$ and $\eta'_{i+1}$ are in adjacent chambers. By assumption, there exist swapping functions $\psi'_i$ between $\eta'_i$ and $\eta'_{i+1}$ such that the recharge for $\eta'_M$ obtained via those swapping operations is
	\[r(\eta'_M,T)-r(\eta_0',T)=\sum_{\alpha\in (\Phi_{n-1})_+}\left(\phi_\alpha(T)-\ell^\alpha(\wt(T))\right)\]

	We can find another sequence of cocharacters $\eta_0,\ldots,\eta_M\in \affX$ such that $\eta_0$ is in the MV chamber and that, for any $i$, the unique wall $H=H_{m\delta-\beta^\vee}$ separating  $\eta_i$ from $\eta_{i+1}$ is the same wall which separates  $\eta'_i$ and $\eta'_{i+1}$. (For example, we can take $\eta_i=\eta'_i+A\varpi_n$ with $A\gg0$.)

	Observe, by \eqref{rMV}, that $r(\eta_0,-)-r(\eta_0',-)$ is constant on the crystal $\calB(\lambda')$. Then, by induction on $i$, we see that $\psi'_i$ is also a swapping function between $\eta_i$ and $\eta_{i+1}$ on $\calB(\lambda')$ and we can extend to a swapping function $\psi_i$ on $\calB(\lambda)$. It follows that for any $T\in \calB(\lambda')$ we have that
	the function
	\[r(\eta_i,T)=r(\eta'_i,T)-r(\eta_0',T)+r(\eta_0,T)\]
	is a recharge for $\eta_i$. The claim follows since $\eta_M$ is in the $(n-1)$-th parabolic region. 
\end{proof}

Therefore, we can assume by induction that we know a recharge for $\eta_P$ in $(n-1)$-th parabolic region. It remains to be understood how to pass from the parabolic to the KL region. This inductive step is carried out in \Cref{inductive} by explicitly constructing  the swapping functions.


Recall that if $\calA$ is an atom of highest weight $\lambda'$, the weights occurring in $\calA$ are precisely the vertices of $\Gamma_{\lambda'}$.

\begin{definition}
		Let $T\in \calB(\lambda)$ and let $\calA$ be the LL atom of highest weight $\lambda'$ containing $T$. 
		For $m\in \bbN \cup \{\infty\}$, we denote by $\Arr_m(T):=\Arr_m(\wt(T),\lambda')$ the number of arrows pointing to $\wt(T)$ in $\Gamma_{\lambda'}^m$.
\end{definition}

By \eqref{Arrm+1} we have
\begin{equation}\label{Arrm+1T}
\Arr_{m+1}(T)=\begin{cases}\Arr_{m}(T)-1& \text{if }t(\wt(T))<\wt(T),\\
\Arr_{m}(T)+1& \text{if }\wt(T)<t(\wt(T))\leq \lambda',\\
\Arr_{m}(T)& \text{if }\wt(T)<t(\wt(T))\not \leq \lambda'\text{ or if }t(\wt(T))=\wt(T),
\end{cases}
\end{equation}
and by \Cref{inftyarrows} we have
\begin{equation}\label{arrinfty} \Arr_{\infty}(T)=\sum_{\beta\in \Phi_+\setminus \Phi_{n-1}} \phi_\beta(T)+\sum_{\beta\in (\Phi_{n-1})_+}\ell^\beta(\wt(T))\end{equation}
In fact, for $\beta\in \Phi_+\setminus \Phi_{n-1}$, if $T$ belongs to an atom of highest weight $\lambda'$, by  \Cref{fandatoms} we have
\[ \max\{ k\in \bbN \mid \wt(T)-k\beta\}= \phi_\beta(T).\]

Applying \eqref{Arrm+1T} and \Cref{atomicdef}, we can restate \Cref{rparabolic}.
\begin{corollary}\label{Arrinf}
		Assume that \Cref{main} holds for $\Phi_{n-1} \cu \Phi$. Let $\eta_P$ be in the $(n-1)$-th parabolic region. Then
		\[r(\eta_P,T)=\shift(T)-\Arr_{\infty}(T)\]
is a recharge for $\eta_P$.
\end{corollary}

\subsection{A Family of Cocharacters}\label{inductive}

In this section we perform the inductive step, constructing a recharge on a crystal $\calB(\lambda)$ in the KL region given a recharge in the $(n-1)$-th parabolic region

The walls that separate the $(n-1)$-th parabolic region from the KL region are precisely
\[ H_{m\delta-\alpha_{k,n}^\vee} \qquad \text{with }m>0\text{ and }1\leq k\leq n.\]

Every cocharacter $\eta_P$ of the form
\[ \eta_P= \sum_{i=1}^n A_i\varpi_i + Cd\]
with $0\ll A_1,A_2, \ldots, A_{n-1}\ll C \ll A_n$ lies in the $(n-1)$-th parabolic region.\footnote{More precise conditions on $\eta_P$ for this to happen are
	\begin{itemize}
		\item $A_1,A_2, \ldots, A_{n-1}>0$,
		\item $C> \sum_{i=1}^{n-1}A_i$,
		\item $A_n> CN$ where $N$ is the largest integer such that $N\delta-\beta^\vee\in \affPhi^\vee(\lambda)$ for some $\beta^\vee \in \Phi^\vee$. 
	\end{itemize}
}

We consider the following family of cocharacters.
\begin{equation}\label{family}
	\eta:\bbQ_{\geq 0}\ra \affX_\bbQ,\qquad \eta(t)=\eta_P+t d.
\end{equation} 
Observe that $\eta(t)$ is in the KL chamber for $t\gg 0$. 
If the integer $N$ is the maximum such that $N\delta-\beta^\vee$ is a root in $\affPhi^\vee(\lambda)$ for some $\beta\in \Phi_+$, then $\eta(t)$ intersects the walls in the following order.
\[H_{N\delta-\alpha_{n}^\vee},H_{N\delta-\alpha_{n-1,n}^\vee},\ldots, H_{N\delta-\alpha_{1,n}^\vee},H_{(N-1)\delta-\alpha_{n}^\vee},\ldots, H_{(N-1)\delta-\alpha_{1,n}^\vee},\ldots, H_{\delta-\alpha_{1,n}^\vee}.\]

Observe that the labels of these walls are precisely the elements of $N(y_M)$, for some $M>0$, although the order in which they occur is the opposite of \eqref{reflectionorder}.

This means that we can choose $0=t_M<t_{M-1}<\ldots t_1<t_0$ such that $\eta(t_i)$ and $\eta(t_{i+1})$ lie in adjacent $\lambda$-chambers for any $i$ with $\eta(t_M)$ lying in the $(n-1)$-th parabolic region and $\eta(t_0)$ in the KL region.\footnote{We enumerate the elements in $t_M,\ldots,t_0$ in this order to make it compatible with the twisted Bruhat graphs $\Gamma_\lambda^m$  of \Cref{twistedbruhatsection}.}

 Set $\eta_m:=\eta(t_m)$ for any $0\leq m \leq M$.
For any $\alpha^\vee\in \affPhi^\vee_+$ we have
 \[ \langle \eta_m,\alpha^\vee\rangle <0 \iff\alpha^\vee \in N(y_m).\]

The only $\lambda$-wall separating $\eta_m$ and $\eta_{m+1}$ is $H_{\alpha_t^\vee}$ where $t=t_{m+1}$ is as in \eqref{tm}.
We prove that modified crystal operators induce the swapping functions and at the same time give a formula for the recharge for any cocharacter $\eta_m$ in the family.

\begin{prop}\label{eswapping2}
Assume that \Cref{main} holds for $\Phi_{n-1}\subset \Phi$.	
Let $0\leq m\leq M$ and let $\alpha_t^\vee=c\delta-\beta^\vee$ for $c>0$ and $\beta^\vee\in \Phi_+$ for $t=t_{m+1}$. 
\begin{enumerate}
	\item The map	$r(\eta_m,T)=\shift(T)-\Arr_{m}(T)$ is a recharge for $\eta_m$.
\item For any $\mu\in X$ such that $\mu<t\mu\leq \lambda$ we define
\[ \psi_{t\mu}:\calB(\lambda)_{t\mu}\ra 
\calB(\lambda)_{\mu},\qquad
\psi_{t\mu}(T)=
e_\beta^{\langle \mu,\beta^\vee	\rangle +c}(T).\]
Then, $\psi=\{\psi_\nu\}$ is a swapping function between $\eta_{m+1}$ and $\eta_m$.
\end{enumerate}
\end{prop}

\begin{proof}

Let $T\in \calB(\lambda)$. Let $\calA$ be an atom of highest weight $\lambda'$ containing $T$. 

Let $M$ be the smallest integer with 
 $\Gamma_{\lambda'}^M =\Gamma_{\lambda'}^\infty$. For $m= M$ the first claim directly follows from \Cref{Arrinf}. We proceed now by reverse induction, assuming  the first statement for $m+1$ and proving both statements for $m$. 
Assume that $\wt(T)=t\mu$. Define
\[\psi_{t\mu}(T):=e_\beta^{\langle\mu,\beta^\vee\rangle+c}(T).\]
We have  $\wt(\psi_{t\mu}(T))=\mu$. Moreover,  $\psi_{t\mu}(T)\in \calA$ since $\beta\not \in \Phi_{n-1}$. By \Cref{Gammam}, we have $\Arr_m(\psi_{t\mu}(T))=\Arr_m(T)-1$.
Recall from \eqref{Arrm+1T} that $\Arr_{m+1}(T)=\Arr_{m}(T)-1$
and that $\Arr_{m+1}(\psi_{t\mu}(T))=\Arr_{m}(\psi_{t\mu}(T))+1$. It follows that
\[\Arr_{m+1}(T)=\Arr_{m+1}(\psi_{t\mu}(T))-1.\]
 The atomic number $Z$ is constant on $\calA$. Therefore, for any $T\in \calB(\lambda)_{t\mu}$, we have
\[ r(\eta_{m+1},\psi_{t\mu}(T))=Z(T)-\Arr_{m+1}(\psi_{t\mu}(T))=r(\eta_{m+1},T)-1.\]
Hence, $\psi$ is a swapping function between $\eta_{m+1}$ and $\eta_{m}$ and the second part is proven for $m+1$.
We now construct $r(\eta_m,-)$ from $r(\eta_{m+1},-)$ and $\psi$ by \eqref{reta2}. 
We can rewrite it in our setting as follows.
\begin{equation}\label{retafinal}
r(\eta_m,T)=\begin{cases}r(\eta_{m+1},T)-1& \text{if }t(\wt(T))<\wt(T),\\
r(\eta_{m+1},T)+1& \text{if }\wt(T)<t(\wt(T))\leq \lambda',\\
r(\eta_{m+1},T)& \text{if }\wt(T)<t(\wt(T))\not \leq \lambda'\text{ or if }t(\wt(T))=\wt(T)
\end{cases}
\end{equation}
Combining \eqref{retafinal} and \eqref{Arrm+1}, it  follows that $r(\eta_m,T)=\shift(T)-\Arr_m(T)$ and the proof is complete.
\end{proof}

\subsection{Proof of the main theorem and consequences}\label{sec:conclusions}

We now have all the ingredients to finally prove our main theorem.

\begin{proof}[Proof of \Cref{main}.]
	We prove the statement by induction on the rank $n$. The case $n=1$ has been established in \cite[\S 3.1]{ChargeGen}. 
	
	Let $\Phi$ the root system of type $A_n$ and assume now by induction that the statement holds in rank $n-1$.
	 In particular, it holds for $\Phi_{n-1}\subset\Phi$, so we can apply \Cref{eswapping2}. We have that $r(\eta_0,T)=Z(T)-\Arr_0(T)$ is a recharge in the KL chamber that can be obtained via swapping operations. 

Let $T$ be contained in an atom of highest weight $\lambda'$. Then $\Gamma^0_{\lambda'}=\Gamma_{\lambda'}$ is the (untwisted) Bruhat graph, so the number of arrows pointing to $\wt(T)$ in $\Gamma^0_{\lambda'}$ is $\ell(\wt(T))$ by \Cref{Xorder}. In other words, we have $\Arr_0(T)=\ell(\wt(T))$ and the proof is complete.
\end{proof}

By \eqref{charge}, we also obtain the following formula for the charge statistic on $\calB(\lambda)$.
\[
c(T)=\shift(T)-\frac12\ell(\wt(T)).\]
In particular, for $\wt(T)\in X_+$, this reduces to
\[c(T)=\sum_{\alpha\in \Phi_+}\eps_\alpha(T).\]

As a consequence, we give a combinatorial formula for the coefficients $a_{\mu,\lambda}(v)$ of the Kazhdan--Lusztig basis in terms of the atomic basis.

Let $\affHec$ denote the spherical Hecke algebra attached to the root system $\Phi\subset X$ (see \cite[\S 2.2]{LPPPre}). Then $\affHec$ has a standard basis $\{ \bfH_{\lambda}\}_{\lambda\in X_+}$ over $\bbZ[v,v^{-1}]$ and a canonical basis, called the Kazhdan--Lusztig basis $\{ \undH_{\lambda}\}_{\lambda\in X_+}$ such that
\[ \undH_\lambda= \sum_{\mu\leq \lambda} K_{\lambda,\mu}(v^2)\bfH_\mu.\]
We also have the \emph{atomic basis} $\{ \bfN_{\lambda}\}_{\lambda \in X_+}$ defined as 
\[\bfN_\lambda =\sum_{\mu\leq \lambda}v^{2\langle \lambda-\mu,\rho^\vee\rangle}\bfH_\mu\]

We can write $\undH_\lambda=\sum a_{\mu,\lambda}(v) \bfN_\mu$. An explicit description of the coefficients $a_{\mu,\lambda}(v)$ has been given in \cite{LPAffine} for type $A_2$ and in \cite{LPPPre} for types $A_3$ and $A_4$.
Here, we provide a combinatorial interpretation of these coefficients for type $A_n$.

\begin{corollary}\label{KLinN}
	Let $\calB(\lambda)=\bigcup_i \calA_i$ be the LL atomic decomposition. Then
	\[\undH_{\lambda}=\sum_{i} v^{2(\shift(\calA_i)-\langle \wt(\calA_i),\rho^\vee\rangle)} \bfN_{\wt(\calA_i)}.\]
	where $\shift(\calA_i)$ is the atomic number of any $T\in \calA_i$ and $\wt(\calA_i)$ is the highest weight of $\calA_i$.
\end{corollary}
\begin{proof}
	
	For any $\lambda \in X_+$, by \Cref{main} we have
	\begin{align*}\undH_\lambda = \sum_{X_+\ni\mu\leq \lambda}h_{\mu,\lambda}(v) \bfH_\mu&=\sum_{T\in \calB^+(\lambda)}v^{2Z(T)-\ell(\wt(\mu))}\bfH_{\wt(T)}\\
		&=\sum_i v^{2\shift(\calA_i)-\ell(\wt(\calA_i))} \sum_{T\in \calA_i \cap \calB^+(\lambda)}v^{\ell(\wt(\calA_i))-\ell(\wt(T))}\bfH_{\wt(T)}.
		\end{align*}
	Since every dominant  weight $\leq \wt(\calA_i)$ occurs exactly once as the weight of an element in $\calA_i$, the claim now follows since
	\[\bfN_{\wt(\calA_i)}=\sum_{X_+\ni\mu \leq \wt(\calA_i)}v^{2\langle \lambda-\mu,\rho^\vee\rangle}\bfH_\mu=\sum_{T\in\calA_i\cap \calB^+(\lambda)}v^{\ell(\wt(\calA_i)-\wt(T)}\bfH_{\wt(T)}. \qedhere\]
	
\end{proof}

\begin{remark}
	We can give another interpretation, arguably more natural, of the formula \eqref{rfinal2} for  the recharge in the KL region.
	Let $\beta\in \Phi_+$. Moving from the MV chamber to the KL chamber we need to cross  all the walls in 
	\begin{equation}\label{betawalls}
		H_{M\delta-\beta^\vee},H_{(M-1)\delta-\beta^\vee},\ldots H_{\delta-\beta^\vee}
	\end{equation}
	in the given order, for some $M\geq 0$.
		Consider a single $\beta$-string: we can think of it as a crystal of type $A_1$, with crystal operator $f_\beta$.
	Each time we cross one of the walls in \eqref{betawalls}, by \Cref{eswapping2}, the recharge on the $\beta$-string varies in the same way as the recharge does on the corresponding crystal of type $A_1$. 
	
	This means that, for $T\in\calB(\lambda)$ and for any $\beta\in \Phi_+$, the contribution of the walls in \eqref{betawalls} to the total variation of the recharge is equal to the total variation of the corresponding $A_1$ crystal. Thanks to \cite[Lemma 3.8]{ChargeGen}, we know that this contribution is
	\[\phi_\beta(T)-\ell^\beta(\wt(T)).\]
	
	To obtain the total variation of the recharge between $\eta$ in the KL chamber and $\eta_0$ in the MV chamber we simply need to add up all these contributions for all $\beta\in \Phi_+$. This gives another explanation of the formula
	\[r(\eta,T)=r(\eta_0,T)+\sum_{\beta\in \Phi_+} \left(\phi_\beta(T)-\ell^\beta(\wt(T))\right).\qedhere\]
	
\end{remark}

\section{Comparison with Lascoux--Sch\"utzenberger's charge}\label{appendix}

In this Appendix we check that the  charge statistic we obtained coincides with Lascoux--Sch\"utzen\-berger's one. We use the formula given by Lascoux, Leclerc and Thibon in \cite{LLTCrystal} which is conveniently already in terms of the crystal graph. Let $d_i(T)=\min (\eps_i(T),\phi_i(T))$.
For $T\in \calB(\lambda)$ they define \[\gamma_n(T)=\frac{1}{(n+1)!}\sum_{\sigma \in W}\sum_{i=1}^n id_i(\sigma(T)).\]
The function $\gamma_n$ is a charge statistic on $\calB(\lambda)$ and it coincides with the charge defined in \cite{LSSur} in terms of the cyclage of tableaux after composing it with the anti-involution of $\calB(\lambda)$ which sends $f_i$ to $e_{n+1-i}$.

\begin{prop}\label{chargecoincide}
Let $n=\rk(\Phi)$. We have $c(T)=\gamma_{n}(T)$ for any $T\in \calB^+(\lambda)$.
\end{prop}
\begin{proof}
	We argue by induction on $n$. The statement is trivial if $n=0$ and the case $n=1$ is handled in detail in \cite[\S 3.1]{ChargeGen}.
	
	In \cite[Theorem 6.5]{LLAtomic} it is shown that Lascoux--Sch\"utzenberger's charge $\gamma$ is compatible with the atomic decomposition, i.e. that if $\calA$ is a LL atom and $T,T'\in \calA$ are such that $\wt(T),\wt(T')\in X_+$ we have
	\[ \gamma_n(T')=\gamma_n(T)+\langle \wt(T)-\wt(T'),\rho^\vee\rangle.\]
	Clearly, also $c$ is compatible with the atomic decomposition. It is enough to check $c(T)=\gamma_n(T)$ for $T$ maximal in its atom.	
	In this case, by \Cref{fandatoms} we have $\eps_\alpha(T)=0$ for any $\alpha\not \in \Phi_{n-1}$, so
	\[c(T)=\sum_{\alpha\in (\Phi_{n-1})_+} \eps_\alpha(T).\]
	
	After Levi branching to $\Phi_{n-1}$, we have by induction $c(T)=\gamma_{n-1}(T)$.
	Moreover, we have $d_n(\sigma(T))=0$ for any $\sigma \in W$. Let $W_{k}\cu W$ be the subgroup generated by $s_1,\ldots,s_k$.
To conclude, it remains to show that
	\begin{equation}\label{key}
		\gamma_n(T)=\frac{1}{n!}\sum_{\sigma\in W_{n-1}}\sum_{i=1}^{n} id_i(\sigma(T))=\gamma_{n-1}(T).
	\end{equation}
We can rewrite \eqref{key} as follows.
\[\label{ugly3}\frac{1}{n}\left(\sum_{j = 1}^n\sum_{\sigma \in W_{n-1}} \sum_{i=1}^{n-1} id_i(\sigma s_n \ldots s_j(T))\right)=\sum_{\sigma\in W_{n-1}}\sum_{i=1}^{n-1} id_i(\sigma (T)).\]
This identity can be easily proven by applying the next Lemma
\end{proof}
	
\begin{lemma}
	For any $T\in \calB(\lambda)$ we have
	\begin{equation}\label{last!}\sum_{\sigma \in W_{n-1}} \sum_{i=1}^{n} id_i(\sigma s_n(T))=\sum_{\sigma \in W_{n-1}} \sum_{i=1}^{n} id_i(\sigma(T)).\end{equation}\end{lemma}
\begin{proof}
	For $n=1$, \eqref{last!} reduces to $d_1(s_1(T))=d_1(T)$, which is clear by definition.
		We assume by induction that \eqref{last!} holds for $n-1$.
	Notice that \[W_{n-1} = W_{n-2}\sqcup \bigsqcup_{j=1}^n W_{n-2}s_{n-1}\cdots s_j \]

By induction we have
\begin{align}\label{ugly1}\sum_{\sigma \in W_{n-1}} \sum_{i=1}^{n-1} id_i(\sigma(T))  &=\sum_{\sigma \in W_{n-2}} \sum_{i=1}^{n-1} id_i(\sigma (T))+ \sum_{j=1}^{n-1}\sum_{\sigma \in W_{n-2}} \sum_{i=1}^{n-1} id_i(\sigma s_{n-1}\cdots s_j (T)) \nonumber\\	
	&=n\sum_{\sigma \in W_{n-2}} \sum_{i=1}^{n-1} id_i(\sigma (T))
\end{align}

Notice that $d_i\circ s_j=d_i$ if $i\neq j\pm 1$. Hence, 
for $i\leq n-2$ we have $d_i(\sigma(T))=d_i(s_n \sigma(T))=d_i(\sigma s_n(T))$ for all $\sigma \in W_{n-2}$. Setting $T':=s_n(T)$, we can rewrite \eqref{ugly1} as
\begin{equation}\label{ugly12}
n(n-1)	\sum_{\sigma \in W_{n-2}} d_{n-1}(\sigma (T))+n\sum_{\sigma \in W_{n-2}} \sum_{i=1}^{n-2}id_{i}(\sigma (T')).
\end{equation}
Using the same argument, we get 
\begin{equation}\label{ugly13}\sum_{\sigma \in W_{n-1}} \sum_{i=1}^{n-1} id_i(\sigma(T'))=n(n-1)	\sum_{\sigma \in W_{n-2}} d_{n-1}(\sigma (T'))+n\sum_{\sigma \in W_{n-2}} \sum_{i=1}^{n-2}id_{i}(\sigma (T')).\end{equation}

Combining \eqref{ugly1}, \eqref{ugly12} and \eqref{ugly13}, we can rewrite the original statement as follows.
\begin{equation}\label{ugly2} \sum_{\sigma \in W_{n-1}} d_n(\sigma T)+(n-1)\sum_{\sigma \in W_{n-2}} d_{n-1}(\sigma T)=\sum_{\sigma \in W_{n-1}} d_n(\sigma T')+(n-1)\sum_{\sigma \in W_{n-2}} d_{n-1}(\sigma T').\end{equation}

We have
\begin{align*} \sum_{\sigma \in W_{n-1}} d_n(\sigma (T))&=\sum_{\sigma \in W_{n-2}}\left(d_n(\sigma T)+\sum_{j=1}^{n-1}d_n(\sigma s_{n-1}\ldots s_jT)\right)\\
	&= (n-1)!\left(d_n(T)+\sum_{j=1}^{n-1}d_n(s_{n-1}\ldots s_jT)\right).
	\end{align*}

Thus, after dividing by $(n-1)!$, the identity \eqref{ugly2} is equivalent to saying that  
\[d_n(T)+\sum_{j=1}^{n-1} d_n(s_{n-1}\ldots s_j (T))+d_{n-1}(T)+\sum_{j=1}^{n-2}d_{n-1}(s_{n-2}\ldots s_j(T))\]
does not change when replacing $T$ with $T'$.
This follows from $d_n(T)=d_n(T')$ and the following claim.
\begin{claim}
	For any $T\in \calB(\lambda)$ we have \[d_n(s_{n-1}(T))+d_{n-1}(T)=d_{n}(s_{n-1}s_{n}(T))+d_{n-1}(s_n(T)).\]
\end{claim}
\begin{proof}[Proof of the claim.]
This is a rank $2$ statement. 
We can thus assume $n=2$ and $r:=s_1s_2s_2=s_{\alpha_1+\alpha_2}$. After replacing $T$ with $s_1(T)$ this is equivalent to 
\[d_2(T)+d_{1}(T)=d_2(r(T))+d_1(r(T)).\]
Recall that $r(T)=s_{\alpha_1+\alpha_2}(T)$ belongs to the $\alpha_1+\alpha_2$-string of $T$. 
Assume that $d_1(T)=\eps_1(T)$ and $d_2(T)=\eps_2(T)$ (the other cases are similar).
Then \[d_2(T)+d_1(T)=\eps_1(T)+\eps_2(T)=\eps_1(T)+\phi_2(T)-\langle \wt(T),\alpha_2^\vee\rangle.\] On the other side, we have 
\begin{align*}d_2(r(T))+d_1(r(T))&=\phi_1(r(T))+\phi_2(r(T))=\eps_1(r(T))+\phi_2(r(T))-\langle \wt(T),r(\alpha_1^\vee)\rangle.\end{align*}
Now the claim follows from \Cref{21explicit}. 
\end{proof}

\end{proof}

\begin{remark}
It was pointed out to the author by M. Shimozono the similarity between our formula for the charge  in \Cref{main} and Nakayashiki--Yamada's description \cite[Corollary 4.2]{NYKostka}. In fact, in both formulas the terms occurring are indexed by positive roots. It seems very plausible (and confirmed by computations on examples in small rank) that the two formulas are actually equivalent, in the sense that they coincide term-wise.  We remark that Nakayashiki--Yamada's formula is of very different nature, as it involves the energy function on tensor products of  Kirillov--Reshetikhin affine crystals and the combinatorial $R$-matrix, which do not play a role in the present paper. We believe it is worth further investigations to clarify the connections between these two descriptions of the charge statistic.
\end{remark}

\def\cprime{$'$}

\bibliography{mybiblio}

\newcommand{\etalchar}[1]{$^{#1}$}
\def\cprime{$'$}
\begin{thebibliography}{BDSW14}

\bibitem[BBD{\etalchar{+}}22]{BBD+Towards}
Charles Blundell, Lars Buesing, Alex Davies, Petar Veli{\v{c}}kovi{\'c}, and
  Geordie Williamson.
\newblock Towards combinatorial invariance for kazhdan-lusztig polynomials.
\newblock {\em Representation Theory of the American Mathematical Society},
  26(37):1145--1191, 2022.

\bibitem[BDSW14]{BDSWNote}
Barbara Baumeister, Matthew Dyer, Christian Stump, and Patrick Wegener.
\newblock A note on the transitive {H}urwitz action on decompositions of
  parabolic {C}oxeter elements.
\newblock {\em Proc. Amer. Math. Soc. Ser. B}, 1:149--154, 2014.

\bibitem[BFG06]{BFGUhlenbeck}
Alexander Braverman, Michael Finkelberg, and Dennis Gaitsgory.
\newblock Uhlenbeck spaces via affine {L}ie algebras.
\newblock In {\em The unity of mathematics}, volume 244 of {\em Progr. Math.},
  pages 17--135. Birkh\"{a}user Boston, Boston, MA, 2006.

\bibitem[BG01]{BGCrystals}
Alexander Braverman and Dennis Gaitsgory.
\newblock Crystals via the affine {G}rassmannian.
\newblock {\em Duke Math. J.}, 107(3):561--575, 2001.

\bibitem[Bry89]{BriLimits}
Ranee~Kathryn Brylinski.
\newblock Limits of weight spaces, {L}usztig's {$q$}-analogs, and fiberings of
  adjoint orbits.
\newblock {\em J. Amer. Math. Soc.}, 2(3):517--533, 1989.

\bibitem[BS17]{BSCrystal}
Daniel Bump and Anne Schilling.
\newblock {\em Crystal bases}.
\newblock World Scientific Publishing Co. Pte. Ltd., Hackensack, NJ, 2017.
\newblock Representations and combinatorics.

\bibitem[Deo85]{DeoSome}
Vinay~V. Deodhar.
\newblock On some geometric aspects of {B}ruhat orderings. {I}. {A} finer
  decomposition of {B}ruhat cells.
\newblock {\em Invent. Math.}, 79(3):499--511, 1985.

\bibitem[Dye92]{DyeHecke}
M.~J. Dyer.
\newblock Hecke algebras and shellings of {B}ruhat intervals. {II}. {T}wisted
  {B}ruhat orders.
\newblock In {\em Kazhdan--{L}usztig theory and related topics ({C}hicago,
  {IL}, 1989)}, volume 139 of {\em Contemp. Math.}, pages 141--165. Amer. Math.
  Soc., Providence, RI, 1992.

\bibitem[Dye19]{DyeWeak}
Matthew Dyer.
\newblock On the weak order of {C}oxeter groups.
\newblock {\em Canad. J. Math.}, 71(2):299--336, 2019.

\bibitem[Kam07]{KamCrystal}
Joel Kamnitzer.
\newblock The crystal structure on the set of {M}irkovi\'{c}-{V}ilonen
  polytopes.
\newblock {\em Adv. Math.}, 215(1):66--93, 2007.

\bibitem[Kas91]{KasCrystal}
M.~Kashiwara.
\newblock On crystal bases of the {$Q$}-analogue of universal enveloping
  algebras.
\newblock {\em Duke Math. J.}, 63(2):465--516, 1991.

\bibitem[LL21]{LLAtomic}
C\'{e}dric Lecouvey and Cristian Lenart.
\newblock Atomic decomposition of characters and crystals.
\newblock {\em Adv. Math.}, 376:51, 2021.

\bibitem[LLT95]{LLTCrystal}
Alain Lascoux, Bernard Leclerc, and Jean-Yves Thibon.
\newblock Crystal graphs and {$q$}-analogues of weight multiplicities for the
  root system {$A_n$}.
\newblock {\em Lett. Math. Phys.}, 35(4):359--374, 1995.

\bibitem[LP23]{LPAffine}
Nicolas Libedinsky and Leonardo Patimo.
\newblock On the affine {H}ecke category for $sl_3$.
\newblock {\em Selecta Mathematica}, 29(4):64, 2023.

\bibitem[LPP22]{LPPPre}
Nicolas Libedinsky, Leonardo Patimo, and David Plaza.
\newblock Pre-canonical bases on affine {H}ecke algebras.
\newblock {\em Adv. Math.}, 399:Paper No. 108255, 2022.

\bibitem[LS78]{LSSur}
Alain Lascoux and Marcel-Paul Sch\"{u}tzenberger.
\newblock Sur une conjecture de {H}. {O}. {F}oulkes.
\newblock {\em C. R. Acad. Sci. Paris S\'{e}r. A-B}, 286(7):A323--A324, 1978.

\bibitem[LS81]{LSLe}
Alain Lascoux and Marcel-P. Sch\"{u}tzenberger.
\newblock Le mono\"{\i}de plaxique.
\newblock In {\em Noncommutative structures in algebra and geometric
  combinatorics ({N}aples, 1978)}, volume 109 of {\em Quad. ``Ricerca Sci.''},
  pages 129--156. CNR, Rome, 1981.

\bibitem[Lus83]{LusSingularities}
George Lusztig.
\newblock Singularities, character formulas, and a {$q$}-analog of weight
  multiplicities.
\newblock In {\em Analysis and topology on singular spaces, {II}, {III}
  ({L}uminy, 1981)}, volume 101 of {\em Ast\'erisque}, pages 208--229. Soc.
  Math. France, Paris, 1983.

\bibitem[Lus90]{LusCanonical}
G.~Lusztig.
\newblock Canonical bases arising from quantized enveloping algebras.
\newblock {\em J. Amer. Math. Soc.}, 3(2):447--498, 1990.

\bibitem[NY97]{NYKostka}
Atsushi Nakayashiki and Yasuhiko Yamada.
\newblock Kostka polynomials and energy functions in solvable lattice models.
\newblock {\em Selecta Math. (N.S.)}, 3(4):547--599, 1997.

\bibitem[Pat21]{PatCombinatorial}
Leonardo Patimo.
\newblock A combinatorial formula for the coefficient of q in
  {K}azhdan--{L}usztig polynomials.
\newblock {\em Int. Math. Res. Not. IMRN}, (5):3203--3223, 2021.

\bibitem[Pat22]{PatBases}
Leonardo Patimo.
\newblock Bases of the intersection cohomology of {G}rassmannian {S}chubert
  varieties.
\newblock {\em J. Algebra}, 589:345--400, 2022.

\bibitem[Pat23]{ChargeGen}
Leonardo Patimo.
\newblock Charges via the affine {G}rassmannian, 2023.

\bibitem[PT23]{PTAtoms}
Leonardo Patimo and Jacinta Torres.
\newblock Atoms and charge in type ${C}_2$, 2023.

\bibitem[Ste98]{StePartial}
John~R. Stembridge.
\newblock The partial order of dominant weights.
\newblock {\em Adv. Math.}, 136(2):340--364, 1998.

\end{thebibliography}
\bibliographystyle{alpha}

\Address

\end{document}